\documentclass[]{interact}

\usepackage{epstopdf}
\usepackage{subfigure}
\usepackage{amssymb}  
\usepackage{psfrag}
\usepackage{mathrsfs}
\usepackage{color}
\usepackage{pgfplots}
\pgfplotsset{compat=newest} 
\newlength\figureheight 
\newlength\figurewidth 
\usepackage{commath} 
\usepackage{tikz}

\usepackage[natbibapa,nodoi]{apacite}
\renewcommand\bibliographytypesize{\fontsize{10}{12}\selectfont}

\theoremstyle{plain}
\newtheorem{theorem}{Theorem}[section]

\theoremstyle{definition}

\theoremstyle{remark}
\newtheorem{rem}{Remark}

\renewcommand{\d}{\mathrm{d}}

\begin{document}


\title{Output feedback control of general linear heterodirectional hyperbolic PDE-ODE systems with spatially-varying coefficients}

\author{
	\name{J.~ Deutscher\textsuperscript{a}\thanks{\emph{Email addresses:} joachim.deutscher@fau.de, nicole.gehring@jku.at and richard.kern@tum.de.}, N.~Gehring\textsuperscript{b} and R.~Kern\textsuperscript{c} }
	\affil{\textsuperscript{a}Lehrstuhl f\"ur Regelungstechnik, Universit\"at Erlangen-N\"urnberg, Cauerstra{\ss}e 7, D--91058 Erlangen, Germany}
	\affil{\textsuperscript{b}Institut f\"ur Regelungstechnik und Prozessautomatisierung, Universit\"at Linz, Altenberger Stra{\ss}e 69, 4040 Linz, Austria}
	\affil{\textsuperscript{c}Lehrstuhl f\"ur Regelungstechnik, Technische Universit\"at M\"unchen,
		Boltzmannstra{\ss}e 15, D-85748 Garching bei M\"unchen, Germany}
}

\maketitle

\begin{abstract}
This paper presents a backstepping solution for the output feedback control of general linear heterodirectional hyperbolic PDE-ODE systems with spatially-varying coefficients. Thereby, the coupling in the PDE is in-domain and at the uncontrolled boundary, whereby the ODE is coupled with the latter boundary. For the state feedback design a two-step backstepping approach is developed, that yields the conventional kernel equations and additional decoupling equations of simple form. The latter can be traced back to simple Volterra integral equations of the second kind, which are directly solvable with a successive approximation. In order to implement the state feedback controller, the design of observers for the ODE-PDE systems in question is considered, whereby anticollocated measurements are assumed. Simple conditions for the existence of the resulting observer-based compensator are formulated, that can be evaluated in terms of the plant transfer behaviour. The resulting systematic compensator design is illustrated for a $4 \times 4$ heterodirectional hyperbolic system coupled with a third order ODE modelling a dynamic boundary condition.
\end{abstract}


\section{Introduction}
In the last decade, the \emph{backstepping approach} emerged as a very powerful tool for stabilizing boundary controlled distributed-parameter systems (DPS) (see, e.\:g., \cite{Kr08} for an overview). The main idea of this method is to introduce invertible Volterra-type integral transformations so that the controller design is facilitated. It was soon recognized that the backstepping approach can also provide systematic solutions for the control of PDE-ODE systems. In the pioneering work \cite{Kr08a} the backstepping method was applied to the stabilization of a \emph{PDE-ODE cascade}, in which the PDE is a first-order hyperbolic system modelling an actuation delay. Afterwards, this result was extended to cascades with diffusion and string PDEs in \cite{Kr09a,Kr09b} enlarging the class of infinite-dimensional actuators.

Subsequently, the backstepping control of \emph{coupled PDE-ODE systems} attracted the interest of many researchers. Thereby, a bidirectional coupling appears between the PDE and the ODE hindering the control design. A first solution for a heat equation coupled with an ODE can be found in \cite{Ta11}. Therein, the coupling in the PDE appears at the unactuated boundary. The same type of coupling was considered in \cite{Sag13} for a wave equation. Another important problem is the stabilization of coupled PDE-ODE systems with an in-domain coupling in the PDE. Solutions for heat equations can be found in \cite{Ta11a}, whereas \cite{Ta12} deals with a wave equation. 

In the last years the backstepping method was extended to a large class of \emph{linear heterodirectional hyperbolic systems}, that consist of an arbitrary number of transport equations convecting in different directions. More precisely, in \cite{Hu15a} the constant coefficient case for this system class is considered. Furthermore, \cite{Hu15b} deals with systems having spatially-varying coefficients, that arise from the linearization of quasilinear hyperbolic systems. These results allow to stabilize more general classes of hyperbolic PDE-ODE systems by making use of the backstepping method. Besides the theoretical appeal of this problem there is also a strong interest originating from applications. Examples are coupled string networks (see, e.\:g., \cite[Ch. 6]{Luo99}), networks of open channels or transmission lines (see, e.\:g., \cite{Bas16}). 

A first solution of the state feedback stabilization problem for coupled linear heterodirectional hyperbolic PDE-ODE systems with constant coefficients was given in \cite{Di16}. Thereby, the coupling between the ODE and the PDE appears at the uncontrolled boundary. The corresponding design is based on a stabilizing backstepping transformation combined with the decoupling transformation to map the plant into the target system. This is also the starting point of all aforementioned results concerning heat and wave equations. Thereby, the decoupling transformation is needed to ensure the decoupling into a PDE-ODE cascade, which is the desired overall target system. As a consequence, the \emph{kernel equations} defining the backstepping transformation and the \emph{decoupling equations} to be solved for the decoupling coordinates are coupled, too. This leads to new kernel equations, which are a system of coupled transport equations connected with a system of coupled ODEs. Hence, a rather involved constructive proof for their solvability is necessary.

Clearly, the results in \cite{Di16} are of great value in itself as they provide a solution method for a rather general class of kernel equations. Nevertheless, an alternative method for the backstepping stabilization of coupled linear heterodirectional hyperbolic PDE-ODE systems is proposed in this paper. The main idea is based on the fact that the stabilization of PDE-ODE system actually leads to two problems. Namely, the backstepping stabilization of the PDE subsystem and the decoupling into a stable PDE-ODE cascade. In order to fully exploit the potential of the backstepping method the PDE subsystem is mapped in the first step into backstepping coordinates. The resulting simple structure of the PDE target system significantly facilitates the decoupling into a PDE-ODE cascade in the second step. As a consequence, only the conventional kernel equations have to be solved, whereby the decoupling equations take a simple form. This \emph{two-step approach} was first proposed for DPS with spatially-varying coefficients in \cite{Deu16a} concerning parabolic PDE-ODE cascades and subsequently for $2 \times 2$ hyperbolic PDE-ODE cascades in \cite{Deu16b}. Therein, the resulting decoupling equations become explicitly solvable in the second step, which leads to a systematic stabilization procedure. These results suggest to extend this method also to linear heterodirectional hyperbolic systems coupled with an ODE.

In this paper the output feedback stabilization of general heterodirectional systems with spatially-varying coefficients coupled with an ODE is considered. Thereby, the coupling of the ODE to the PDE subsystem appears at the uncontrolled boundary and in-domain in the PDE, whereas the ODE is subject to a coupling with the uncontrolled boundary. The two-step approach is extended to the considered class of PDE-ODE systems in order to design  the state feedback controller. More precisely, in the first step a backstepping transformation is utilized to map the hyperbolic subsystem into its target system, which is a cascade of transport equations. For this, only the conventional kernel equations found in \cite{Hu15b} have to be solved. This significantly simplifies the calculation of the decoupling coordinates in the second step. Thereby, the inverse decoupling transformation is determined, because the related decoupling equations have a very simple structure when compared to the direct transformation. In particular, as decoupling equations a set of coupled ODEs and a set of decoupled transport equations are obtained. Thereby, the ODEs are not coupled with the PDEs so that an explicit solution of the former is possible. Furthermore, it is shown that the solution of the PDEs can be traced back to solving  $p^2$ simple scalar \emph{Volterra integral equations of the second kind}, if the plant has $p$ inputs. This  results in a systematic method for determining the decoupling coordinates, because the solution of the related integral equations can readily be obtained from utilizing a truncated fixpoint iteration. Subsequently, the obtained state feedback controller is implemented with an observer. By assuming anticollocated measurements a systematic method is proposed for the corresponding PDE-ODE observer design. This extends recent results concerning observers for general heterodirectional PDE-ODE cascades with constant coefficients in \cite{An16} and in \cite{Deu16c} for the spatially-varying case. For the existence of the corresponding PDE-ODE observers and thus of the resulting observer-based compensator simple conditions are presented in terms of the transfer behaviour w.r.t. the infinite-dimensional subsystem. This yields a systematic backstepping method to the output feedback control for a large class of coupled linear heterodirectional PDE-ODE systems.

The next section introduces the considered stabilization problem. Then, the two-step approach for the state feedback design is presented in Section \ref{sec:statefeed}. A systematic solution of the decoupling equations is given in the subsequent section. By assuming anticollocated measurements the observer design is presented in Section \ref{sec:obs}. The theoretical part of the paper is concluded with the proof of closed-loop stability in Section \ref{sec:clstab}. A $4 \times 4$ heterodirectional hyperbolic system with a dynamic boundary condition is utilized to demonstrate the results of the article. 

\section{Problem formulation}\label{sec:probform}
Consider the \emph{general linear hyperbolic PDE-ODE system}
\begin{subequations}\label{plant}
	\begin{align}
	\partial_tx(z,t) &= \Lambda(z)\partial_zx(z,t) \!+\! A(z)x(z,t) \!+\! C_1(z)\xi(t), && (z,t) \in (0,1) \times \mathbb{R}^+ \label{xeq}\\
	x_2(0,t) &= Q_0x_1(0,t) +  C_2\xi(t), &&  t > 0\label{uabc}\\
	x_1(1,t) &= Q_1x_2(1,t) + u(t), &&  t > 0\label{aktbc}\\
	\dot{\xi}(t) &= F\xi(t) + Bx_1(0,t), && t > 0\label{plantode}\\
        	y(t) &= x_1(0,t), &&  t \geq 0,\label{meas}
	\end{align}
\end{subequations}
that consists of  $n$ coupled \emph{transport PDEs} \eqref{xeq} with the distributed state $x(z,t) = [x^1(z,t) \;\; \ldots \;\; x^n(z,t)]^T \in \mathbb{R}^{n}$, the ODE \eqref{plantode} with the lumped state $\xi(t) \in \mathbb{R}^{n_{\xi}}$, the input $u(t) \in \mathbb{R}^p$  and the \emph{anticollocated measurement} $y(t) \in \mathbb{R}^p$. Furthermore, $Q_0 \in \mathbb{R}^{m \times p}$ and $Q_1 \in \mathbb{R}^{p \times m}$ with $p+m = n$ and $p, m \geq 1$ are arbitrary matrices and $\Lambda(z)$ in \eqref{xeq} is given by
\begin{align}\label{Lamdef}
\Lambda(z)  &= \operatorname{diag}(\lambda_1(z),\ldots,\lambda_{n}(z))\nonumber\\
&=\operatorname{diag}(\epsilon_1(z),\ldots,\epsilon_p(z),-\epsilon_{p+1}(z),
\ldots,-\epsilon_{n}(z))
\end{align}
where $\epsilon_i \in C^1[0,1]$, $i = 1,2,\ldots,n$, and $\epsilon_1(z) > \ldots > \epsilon_{p}(z) > 0 > -\epsilon_{p+1}(z) > \ldots > -\epsilon_{n}(z)$, $z \in [0,1]$.  Moreover, the matrix $A(z) = [A_{ij}(z)]$ in \eqref{xeq} satisfies $A_{ii}(z) = 0$, $z \in [0,1]$, $i = 1,2,\ldots,n$, and $A_{ij} \in C^1[0,1]$, $i,j = 1,2,\ldots,n$,  $C_1 \in (L_2(0,1))^{n \times n_{\xi}}$ and $C_2 \in \mathbb{R}^{m \times n_{\xi}}$.

In order to obtain a compact representation of the results, the matrices 
\begin{equation}\label{Edef}
E_1 = \begin{bmatrix}
I_p\\
0
\end{bmatrix} \in \mathbb{R}^{n \times p}\quad \text{and} \quad 
E_2 = \begin{bmatrix}
0\\
I_m
\end{bmatrix} \in \mathbb{R}^{n \times m}
\end{equation}
are defined. With this, the states $x_1(z,t) = E_1^Tx(z,t) \in \mathbb{R}^p$ describe the convection in the negative direction of the spatial coordinate $z$ with the velocities $\epsilon_i(z)$, $i = 1,2,\ldots,p$. The remaining states $x_2(z,t) = E_2^Tx(z,t) \in \mathbb{R}^m$ with the velocities $\epsilon_{i}(z)$, $i = p+1,\ldots,n$, take the convection in the $z$-direction into account. Hence, the distributed-parameter subsystem \eqref{xeq}--\eqref{aktbc} is a \emph{heterodirectional system} (see \cite{Hu15a}). The matrix pair $(F,B)$ characterizing the ODE \eqref{plantode} with $F \in \mathbb{R}^{n_{\xi} \times n_{\xi}}$ and $B \in \mathbb{R}^{n_{\xi} \times p}$ is assumed to be stabilizable. Finally, the \emph{initial conditions (IC)} of \eqref{plant} are $x(z,0) = x_0(z) \in \mathbb{R}^{n}$, $z \in [0,1]$, and $\xi(0) = \xi_0 \in \mathbb{R}^{n_{\xi}}$.

\begin{rem}
The considered plant \eqref{plant} comprises \emph{bidirectionally coupled PDE-ODE systems} (i.\:e., at least one of the $C_i$, $i = 1,2$, is not vanishing). An important example is the case $C_1(z) = 0$ and $C_2 \neq 0$, which  appears if the DPS is coupled with a lumped-parameter system at the unactuated boundary $z = 0$. This gives rise to a \emph{dynamic boundary condition (BC)}. Finally, for $C_i = 0$, $i = 1,2$, a \emph{PDE-ODE cascade} is obtained.
	\hfill $\triangleleft$
\end{rem}

\begin{rem}
	It should be noted that the assumed form of \eqref{xeq} can always be obtained from the general case, in which $\Lambda(z)$ and $A(z)$ are arbitrary matrices with elements in $C^1[0,1]$. This is possible by making use of the transformations given in \cite[Ch. 6.9]{Deb12,Vaz14,La15}. To this end, the corresponding system has to be \emph{strictly hyperbolic} and heterodirectional. \hfill $\triangleleft$
\end{rem}


This paper concerns the \emph{backstepping design} of a \emph{compensator}
	that stabilizes the resulting closed-loop system.

\section{State feedback design}\label{sec:statefeed}
Consider the \emph{state feedback controller}
	\begin{equation}\label{sfeed}
	u(t) =  - Q_1x_2(1,t) + \mathcal{K}[\xi(t),x(t)]
	\end{equation}
	with the formal \emph{feedback operator}
	\begin{equation}\label{fop}
	\mathcal{K}[\xi(t),x(t)] = -K_{\xi}\xi(t) -\textstyle\int_0^1K_x(z)x(z,t)\d z.
	\end{equation}
By inserting \eqref{sfeed} into \eqref{plant} the \emph{closed-loop system}
\begin{subequations}\label{plant2}
	\begin{align}
	\partial_tx(z,t) &= \Lambda(z)\partial_zx(z,t) + A(z)x(z,t) + C_1(z)\xi(t) \label{PDE2}\\
	x_2(0,t) &= Q_0x_1(0,t) + C_2\xi(t)\label{rb12}\\     
	x_1(1,t) &= \mathcal{K}[\xi(t),x(t)]\label{rb22}\\
	\dot{\xi}(t) &= F\xi(t) + Bx_1(0,t)\label{ODE2}
	\end{align}
\end{subequations}
is obtained. This is a PDE-ODE system with a bidirectional coupling in \eqref{rb12} and \eqref{ODE2}. Hence, in order to apply the backstepping method for the design of the state feedback \eqref{sfeed} a decoupling of the closed-loop system into a \emph{PDE-ODE cascade} has to be considered. Then, the controller can be derived from the stabilization of the resulting PDE and ODE subsystems. The calculation of the corresponding decoupling transformation can be significantly simplified if the closed-loop system \eqref{plant2} is mapped into backstepping coordinates. This is shown in the next section.

\subsection{Backstepping Transformation}
Consider the invertible \emph{backstepping transformation}
\begin{align}\label{btrafo}
\tilde{x}(z,t) = x(z,t) - \textstyle\int_0^zK(z,\zeta)x(\zeta,t)\d \zeta = \mathcal{T}_1[x(t)](z)
\end{align}
with the \emph{integral kernel} $K(z,\zeta) \in \mathbb{R}^{n \times n}$ (see \cite{Hu15a,Hu15b}). It is assumed that $K(z,\zeta)$ is the solution of the \emph{kernel equations} 
\begin{subequations}\label{ckbvp}
	\begin{align}
	\hspace{-0.3cm}\Lambda(z)\partial_zK(z,\zeta) + \partial_{\zeta}(K(z,\zeta)\Lambda(\zeta)) &= K(z,\zeta)A(\zeta), \quad 0 < \zeta < z < 1 \label{cdbvp1}\\
	K(z,0)\Lambda(0)(E_1 + E_2Q_0) &= A_0 (z)\label{cdbvp3}\\
	K(z,z)\Lambda(z) - \Lambda(z) K(z,z) &= A(z).\label{cdbvp2}
	\end{align}
\end{subequations}
Therein, the matrix $A_0(z)$ is given by
\begin{equation}\label{A0def}
A_0(z) = \begin{bmatrix}
A_{1}(z)\\ 
A_{2}(z)
\end{bmatrix},
\end{equation}
in which $A_{1}(z) \in \mathbb{R}^{p \times p}$ is \emph{strictly lower triangular}, i.\:e.,
\begin{equation}\label{a01def}
A_{1}(z) = \begin{bmatrix} 0         & \ldots &   \ldots          & 0\\
a_{21}(z) & \ddots       &  \ddots   & \vdots\\
\vdots    & \ddots & \ddots            & \vdots \\
a_{p1}(z) & \ldots & a_{p\,p-1}(z) & 0
\end{bmatrix}
\end{equation}
and $A_{2}(z) \in \mathbb{R}^{m \times p}$ has no special form. Thereby, the elements of the strictly lower triangular part of \eqref{a01def} and of $A_2(z)$ are determined by the kernel (for details see the Appendix \ref{appA}). With the \emph{method of characteristics} the boundary value problem (BVP) \eqref{ckbvp} can be converted into integral equations. The latter are solvable by means of a fixpoint iteration. This allows to show that a unique piecewise $C^1$-solution $K(z,\zeta)$ of \eqref{ckbvp} exists (see \cite{Hu15b}). Hence, the kernel is attainable by utilizing a \emph{successive approximation}. Further details for solving \eqref{ckbvp} are provided in the Appendix \ref{appA}. 

Differentiating \eqref{btrafo} w.r.t. time, utilizing \eqref{plant}, \eqref{fop}  \eqref{ckbvp} and $\tilde{x}_i = E_i^T\tilde{x}$, $i = 1,2$, results in the closed-loop system 
\begin{subequations}\label{plant3}
	\begin{align}
	\partial_t\tilde{x}(z,t) &= \Lambda(z)\partial_z\tilde{x}(z,t) + A_0(z)\tilde{x}_1(0,t)  + G(z)\xi(t),&& (z,t) \in (0,1) \times \mathbb{R}^+\label{PDE3}\\
	\tilde{x}_2(0,t) &= Q_0\tilde{x}_1(0,t) + C_2\xi(t),&& t > 0\label{rb13}\\     
	\tilde{x}_1(1,t) &= \mathcal{K}[\xi(t),x(t)] \!- \! \textstyle\int_0^1E_1^TK(1,z)x(z,t)\d z,&& t > 0\label{rb23}\\
	\dot{\xi}(t) &= F\xi(t) + B\tilde{x}_1(0,t),&& t > 0\label{ODE3}
	\end{align}
\end{subequations}
in backstepping coordinates. Therein, the matrix $G(z)$ in \eqref{PDE3} follows as
\begin{equation}\label{Gdef}
G(z) = K(z,0)\Lambda(0)E_2C_2 + C_1(z) - \textstyle\int_0^zK(z,\zeta)C_1(\zeta)\d\zeta.
\end{equation}

\subsection{Decoupling into a PDE-ODE cascade}\label{sec:pdeodedecoupl}
There exist in principle two possibilities for the decoupling of \eqref{plant3}: the decoupling of the ODE subsystem \eqref{ODE3} or the decoupling of the PDE subsystem \eqref{PDE3}--\eqref{rb23} from the corresponding subsystem. It can be shown that the former approach leads to decoupling equations, which are not solvable. Hence, one has to pursue the second approach. For this, introduce the \emph{decoupling coordinates} $e_x$ in form of 
the \emph{inverse decoupling transformation}
\begin{equation}\label{ccord2}
\tilde{x}(z,t) =  \mathcal{T}^{-1}_2[e_x(t)](z) + N_I(z)\xi(t).
\end{equation}
Therein, 
\begin{equation}\label{invT2}
\mathcal{T}_2^{-1}[e_x(t)](z) = e_x(z,t) + \textstyle\int_0^zP_I(z,\zeta)e_x(\zeta,t)\d \zeta
\end{equation}
represents the inverse Volterra-type transformation with the integral kernel $P_I(z,\zeta) \in \mathbb{R}^{n \times n}$ w.r.t. the transformation
\begin{align}\label{T2}
\mathcal{T}_2[\tilde{x}(t)](z) = \tilde{x}(z,t) - \textstyle\int_0^zP(z,\zeta)\tilde{x}(\zeta,t)\d \zeta
\end{align}
with the integral kernel $P(z,\zeta) \in \mathbb{R}^{n \times n}$  and $N_I(z) \in \mathbb{R}^{n  \times n_{\xi}}$ holds. 
\begin{rem}
	The direct transformation 
	\begin{equation}\label{ccord}
	e_x(z,t) =  \mathcal{T}_2[\tilde{x}(t)](z) - N(z)\xi(t)
	\end{equation}
	yields \emph{decoupling equations} to be fulfilled by $P(z,\zeta)$ and $N(z) = \mathcal{T}_2[N_I](z)$ that consists of a BVP mutually coupled with an initial value problem (IVP). Their solution can also be traced back to solving Volterra integral equations of the second kind. However, due to the coupling between the BVP and IVP the corresponding kernels become very involved for $p > 2$ and thus are hard to determine in general. As will be shown in the sequel, the corresponding \emph{inverse decoupling equations} determining \eqref{ccord2} have a significantly simpler structure, which leads to a systematic method for the decoupling into a PDE-ODE cascade.  \hfill $\triangleleft$
\end{rem}

For the derivation of the inverse decoupling equations the \emph{target system} in form of the \emph{PDE-ODE cascade} 
\begin{subequations}\label{pde-ode cascade}
	\begin{align}
	\partial_te_x(z,t) &= \Lambda(z)\partial_ze_x(z,t) + H_0(z)e_{x_1}(0,t), &&(z,t) \in (0,1) \times \mathbb{R}^+ \label{casdpde}\\
	e_{x_2}(0,t) &= Q_0e_{x_1}(0,t),&& t > 0\label{cascpderb1}\\
	e_{x_1}(1,t) &= 0, &&t > 0\label{cascrb2}\\
	\dot{\xi}(t) &= (F - BK)\xi(t) + Be_{x_1}(0,t), && t > 0\label{cascode}
	\end{align}	 
\end{subequations}
and $e_{x_i} = E_i^Te_x$, $i = 1,2$, is proposed. Therein, the matrix 
\begin{equation}\label{H0def}
H_0(z) = \begin{bmatrix}
H_{1}(z)\\ 
H_{2}(z) 
\end{bmatrix} \in \mathbb{R}^{n \times p}
\end{equation}
in \eqref{casdpde} with
\begin{equation}\label{h01def}
H_{1}(z) = \begin{bmatrix} 0         & \ldots &  \ldots           & 0\\
h_{21}(z) & \ddots &  \ddots         & \vdots\\
\vdots    & \ddots &  \ddots     & \vdots \\
h_{p1}(z) & \ldots & h_{p\,p-1}(z) & 0
\end{bmatrix} \in \mathbb{R}^{p \times p}
\end{equation}
and $H_2(z) \in \mathbb{R}^{m \times p}$ has to be introduced in order to ensure well-posedness of the resulting inverse decoupling equations (see the next section). 
By making use of the results in \cite{Hu15b} it follows that the PDE-subsystem \eqref{casdpde}--\eqref{cascrb2} is \emph{finite-time stable} for piecewise continuous IC, i.\:e., 
\begin{equation}\label{eq:tc}
 e_x(z,t) = 0 , \quad t \geq t_c = \sum_{i=1}^{p+1}|\phi_i(1)|. 
\end{equation}	
Hence, if $K$ ensures that $F - BK$ is Hurwitz, then the target system \eqref{pde-ode cascade} is asymptotically stable. Such a feedback gain $K$ always exists, as $(F,B)$ is stabilizable by assumption.

Differentiating \eqref{ccord2} w.r.t. time and inserting \eqref{casdpde} and \eqref{ODE3} leads  to
\begin{align}\label{expre}
\partial_t\tilde{x}(z,t) &= \partial_te_x(z,t)  + \textstyle\int_0^zP_I(z,\zeta)\partial_te_x(\zeta,t)\d \zeta + N_I(z)\dot{\xi}(t)\nonumber\\
&= \Lambda(z)\partial_z\tilde{x}(z,t) + A_0(z)\tilde{x}_1(0,t) + G(z)\xi(t)\nonumber\\
&\quad  + (\mathcal{T}^{-1}_2[H_0](z) - A_0(z)+ N_I(z)B)e_{x_1}(0,t)\nonumber\\
&\quad + (N_I(z)(F - BK) - G(z) - \Lambda(z)N'_I(z)\nonumber\\
&\quad - A_0(z)E_1^TN_I(0))\xi(t) - \Lambda(z)\partial_z\textstyle\int_0^zP_I(z,\zeta)e_x(\zeta,t)\d\zeta\nonumber\\
&\quad + \textstyle\int_0^zP_I(z,\zeta)\Lambda(\zeta)\partial_{\zeta}e_x(\zeta,t)\d\zeta. \end{align}
By utilizing an integration by parts and applying the Leibniz differentiation rule one obtains
\begin{align}\label{expre2}
\partial_t\tilde{x}(z,t) &= \Lambda(z)\partial_z\tilde{x}(z,t) + A_0(z)\tilde{x}_1(0,t) + G(z)\xi(t)\nonumber\\
&\quad  - \textstyle\int_0^z(\Lambda(z)\partial_zP_I(z,\zeta) + \partial_{\zeta}(P_I(z,\zeta)\Lambda(\zeta)))e_x(\zeta,t)\d\zeta\nonumber\\
&\quad  + (N_I(z)B - A_0(z) + \mathcal{T}^{-1}_2[H_0](z)\nonumber\\
&\quad    - P_I(z,0)\Lambda(0)(E_1 + E_2Q_0))e_{x_1}(0,t)\nonumber\\
&\quad  + (P_I(z,z)\Lambda(z) - \Lambda(z)P_I(z,z))e_x(z,t)\nonumber\\
&\quad  + (N_I(z)(F-BK) - \Lambda(z)N_I'(z) - A_0(z)E_1^TN_I(0)\nonumber\\
&\quad  - G(z))\xi(t).
\end{align}
Evaluating \eqref{ccord2} at $z = 0$, taking \eqref{rb13} and \eqref{cascpderb1} into account yields
\begin{align}\label{rbz=0}
\tilde{x}_2(0,t) &= e_{x_2}(0,t) + E_2^TN_I(0)\xi(t)\nonumber\\ 
&= Q_0e_{x_1}(0,t) + E_2^TN_I(0)\xi(t)\nonumber\\
&= Q_0e_{x_1}(0,t) + (Q_0E_1^TN_I(0) + C_2)\xi(t),
\end{align}
which implies $E_2^TN_I(0) = Q_0E_1^TN_I(0) + C_2$. Inserting \eqref{ccord2} in \eqref{ODE3} gives
\begin{align}\label{ODEdecoupl}
\dot{\xi}(t) = (F + BE_1^TN_I(0))\xi(t) + Be_{x_1}(0,t)
\end{align}
so that $E_1^TN_I(0) = -K$ follows in light of \eqref{cascode}. In order to determine the feedback operator \eqref{fop}, consider \eqref{ccord2} at $z = 1$ and utilize \eqref{rb23} giving
\begin{align}\label{rbz=1}
\tilde{x}_1(1,t) &= e_{x_1}(1,t) + \textstyle\int_0^1E_1^TP_I(1,z)e_x(z,t)\d z
 + E_1^TN_I(1)\xi(t)\nonumber\\ 
&=  \mathcal{K}[\xi(t),x(t)] - \textstyle\int_0^1E_1^TK(1,z)x(z,t)\d z.
\end{align}
Consequently, \eqref{cascrb2} implies
\begin{equation}\label{contcalc}
\mathcal{K}[\xi(t),x(t)] = \textstyle\int_0^1E_1^TK(1,z)x(z,t)\d z
 + \textstyle\int_0^1E_1^TP_I(1,z)e_x(z,t)\d z
  + E_1^TN_I(1)\xi(t).
\end{equation} 
Hence, it can be directly inferred from \eqref{expre2}--\eqref{ODEdecoupl} that the closed-loop system \eqref{plant3} is mapped into the PDE-ODE cascade \eqref{pde-ode cascade}
if $N_I(z)$ and $P_I(z,\zeta)$ are the solution of the \emph{inverse decoupling equations}
\begin{subequations}\label{decouplODE}
	\begin{align}
	\Lambda(z)N'_I(z) - N_I(z)(F-BK) &= - A_0(z)E_1^TN_I(0) - G(z), \quad z \in (0,1)\label{NODE}\\
	(E_2^T - Q_0E_1^T)N_I(0) &= C_2\label{Kbed1}\\
	E_1^TN_I(0) &= -K\label{Kbed}
	\end{align}
\end{subequations}
and 
\begin{subequations}\label{decouplBVP}
	\begin{align}
	\Lambda(z)\partial_zP_I(z,\zeta) + \partial_{\zeta}(P_I(z,\zeta)\Lambda(\zeta)) &= 0, \quad 0 < \zeta < z < 1\label{couplpde}\\
	P_I(z,0)\Lambda(0)(E_1 + E_2Q_0) &= N_I(z)B - A_0(z)+ \mathcal{T}^{-1}_2[H_0](z)\label{BCcomp}\\
	P_I(z,z)\Lambda(z) - \Lambda(z)P_I(z,z) &= 0.\label{comu}
	\end{align}
\end{subequations}
\begin{rem}	
The simple form of the PDE \eqref{couplpde} results from considering the decoupling  in the backstepping coordinates. In particular, this PDE consists of $n$ decoupled transport equations. If a decoupling in the original coordinates $x(z,t)$ is considered, i.\:e., the backstepping transformation \eqref{btrafo} is combined with the decoupling coordinates \eqref{ccord}, then one obtains a PDE of the form \eqref{cdbvp1}. Hence, one has to solve a set of coupled transport equations with spatially-varying coefficients, that is additionally coupled with the IVP for $N(z)$. This would render the solution of the combined kernel and decoupling equations much more difficult. Furthermore, the consideration of the inverse decoupling equations yields the IVP \eqref{decouplODE} for $N_I(z)$, which is decoupled from the BVP \eqref{decouplBVP}. This further significantly simplifies the derivation of the corresponding Volterra integral equations to be determined for solving \eqref{decouplBVP} (see next section). \hfill $\triangleleft$
\end{rem}

In order to make \eqref{contcalc} more explicit, the \emph{reciprocity relations} for \eqref{invT2} and \eqref{T2} are derived. Inserting  \eqref{T2} in  \eqref{invT2} and changing the order of integration yields the \emph{reciprocity relation} 
\begin{equation}\label{recrel}
P(z,\zeta) + \textstyle\int_{\zeta}^{z}P_I(z,\zeta')P(\zeta',\zeta)\d\zeta' = P_I(z,\zeta).
\end{equation}  
This is a Volterra integral equation for the kernel $P(z,\zeta)$ of \eqref{T2}. Hence, if $P_I(z,\zeta)$ is well-defined and bounded so is $P(z,\zeta)$ (see, e.\:g., \cite[Th. 3.4]{Lin85}) implying that $\mathcal{T}_2[\cdot]$ exists. Similarly, by substituting \eqref{invT2} in \eqref{T2} and changing the order of integration one obtains the alternative \emph{reciprocity relation} 
\begin{equation}\label{recrelalt}
P(z,\zeta) + \textstyle\int_{\zeta}^{z}P(z,\zeta')P_I(\zeta',\zeta)\d\zeta' = P_I(z,\zeta).
\end{equation}  
In order to simplify \eqref{contcalc} insert \eqref{ccord} in \eqref{contcalc}, use \eqref{T2}, change the order of integration and utilize \eqref{recrel} to obtain
\begin{align}\label{contcalc2}
\mathcal{K}[\xi(t),x(t)]
&= \textstyle\int_0^1E_1^T(K(1,z) + P(1,z)
  - \textstyle\int_z^1P(1,\zeta)K(\zeta,z)\d\zeta)x(z,t)\d z\nonumber\\
&\quad +  E_1^T(N_I(1) - \textstyle\int_0^1P(1,z)N_I(z)\d z)\xi(t).
\end{align}
From this, the feedback gains
\begin{subequations}\label{gains}
	\begin{align}
	K_{\xi} &= E_1^T(\textstyle\int_0^1P(1,z)N_I(z)\d z - N_I(1))\\
	K_x(z) &= E_1^T(\textstyle\int_z^1P(1,\zeta)K(\zeta,z)\d\zeta - K(1,z) - P(1,z))
	\end{align}
\end{subequations}
can be directly deduced in view of \eqref{fop}. 

Both reciprocity relations \eqref{recrel} or \eqref{recrelalt} can be utilized to calculate $P(z,\zeta)$. However, the result \eqref{gains} shows that for determining the feedback operator \eqref{contcalc2} only $P(1,\zeta)$ has to be known. In order to get $P(1,\zeta)$ set $z = 1$ in \eqref{recrelalt}. This results in the \emph{Volterra integral equation of the second kind}
\begin{equation}\label{recrelp1z}
P(1,\zeta) + \textstyle\int_{\zeta}^{1}P(1,\zeta')P_I(\zeta',\zeta)\d\zeta' = P_I(1,\zeta)
\end{equation}
for $P(1,\zeta)$. Hence, the feedback gains  \eqref{gains} can be derived from the solution $P_I(z,\zeta)$ and $N_I(z)$ of the inverse decoupling equations \eqref{decouplODE} and \eqref{decouplBVP}.

\section{Solution of the Inverse Decoupling Equations}\label{sec:soldecoupl}
The fact that the IVP \eqref{decouplODE} is decoupled from the BVP \eqref{decouplBVP} allows to determine the solution of $N_I(z)$ independently. Subsequently, the resulting $N_I(z)$ can be utilized for solving \eqref{decouplBVP}. 

\subsection{Solution of the IVP for $N_I(z)$}\label{sec:N_isolve}
For the solution of the IVP \eqref{decouplODE}  it is convenient to split the corresponding equations into two IVP by defining
\begin{equation}\label{Nidef}
M_i(z) = E_i^TN_I(z), \quad i = 1,2.
\end{equation}
In order to obtain the ODEs for \eqref{Nidef} write \eqref{NODE}  in the form
\begin{equation}\label{NODE2} 
N_I'(z) = \Lambda^{-1}(z)(N_I(z)(F-BK) + A_0(z)K - G(z)),
\end{equation}
in which \eqref{Kbed} has been inserted. Then, premultiply \eqref{NODE2} by $E_1E_1^T + E_2E_2^T = I_n$ and utilize $E_i^T\Lambda^{-1}(z) = \Lambda_i^{-1}(z)E_i^T$, $i = 1,2$, with $\Lambda_i^{-1}(z) = E_i^T\Lambda^{-1}(z)E_i$.  This yields
\begin{multline}\label{Niodes}
E_1M'_1(z) + E_2M'_2(z)
= E_1\Lambda^{-1}_1(z)(M_1(z)(F-BK) + A_1(z)K - E_1^TG(z))\\
\quad + E_2\Lambda^{-1}_2(z)(M_2(z)(F-BK) + A_2(z)K - E_2^TG(z)).
\end{multline}
Hence, equating \eqref{Niodes} w.r.t. $E_1$ and $E_2$ as well as taking \eqref{Kbed1} and \eqref{Kbed} into account leads to the IVP
\begin{subequations}\label{ivpN1}
	\begin{align}
	M_1'(z) &= \Lambda^{-1}_1(z)(M_1(z)(F-BK) + A_1(z)K - E_1^TG(z)), \quad z \in (0,1)\label{N1ode}\\
	M_1(0) &= -K
	\end{align}
\end{subequations}
and 
\begin{subequations}\label{ivpN2}
	\begin{align}
	M_2'(z) &= \Lambda^{-1}_2(z)(M_2(z)(F-BK) + A_2(z)K - E_2^TG(z)), \quad z \in (0,1)\label{N2ode}\\
	M_2(0) &= C_2 - Q_0K.\label{N2bc}
	\end{align}
\end{subequations}

In order to solve \eqref{ivpN1} and \eqref{ivpN2} the \emph{fundamental matrices} $\Psi_i(z,\zeta): [0,1]^2 \to \mathbb{R}^{n_{\xi} \times n_{\xi}}$, $i = 1,2,\ldots,n$, resulting from the IVP
\begin{equation}\label{psiivp}
\partial_z\Psi_i(z,\zeta) = \textstyle\frac{1}{\lambda_i(z)}(F-BK)\Psi_i(z,\zeta), \quad \Psi_i(\zeta,\zeta) = I
\end{equation}
are needed. It is easy to show that the solution of \eqref{psiivp} is
\begin{equation}\label{fmatdef}
\Psi_i(z,\zeta) = \text{e}^{(F-BK)(\phi_i(z) - \phi_i(\zeta))}
\end{equation}
with
\begin{align}\label{phidef}
\phi_i(z) = \textstyle\int_0^{z}\textstyle\frac{\d\zeta}{\lambda_i(\zeta)},\quad z \in [0,1], \quad i = 1,2,\ldots,n.
\end{align}
For determining $M_1(z)$ premultiply \eqref{ivpN1} by $e_i^T$ with $e_i$ denoting the $i$-th unit vector in $\mathbb{R}^p$ and introduce the abbreviation $\nu_i^T(z) = e_i^TM_1(z)$, $i = 1,2,\ldots,p$. With this, one obtains the set of IVP
\begin{subequations}\label{ivpN2setm}
	\begin{align}
	\d_z\nu_i^T(z) &= \textstyle\frac{1}{\lambda_i(z)}\nu_i^T(z)(F-BK)\label{N2odem}
 + \textstyle\frac{1}{\lambda_i(z)}e_i^T(A_1(z)K - E_1^TG(z)), \quad z \in (0,1)\\
	\nu_i^T(0) &= -e_i^TK.\label{N2bcm}
	\end{align}
\end{subequations}
Their solution is 
\begin{equation}
\nu_i^T(z) = -e_i^TK\Psi_i(z,0) + \textstyle\int_0^z\frac{1}{\lambda_i(\zeta)}e_i^T(A_1(\zeta)K - E_1^TG(\zeta))\Psi_i(z,\zeta)\d\zeta
\end{equation}
for $i = 1,2,\ldots,p$ in view of \eqref{fmatdef}. From this, the piecewise classical solution $M_1(z)$ of \eqref{ivpN1} directly follows, because $G(z)$ is piecewise continuous (see \eqref{Gdef} and the Appendix \ref{appA}). Similarly, by defining $\mu_i^T(z) = e_i^TM_2(z)$, $i = 1,2,\ldots,m$, with the unit vector $e_i \in \mathbb{R}^m$ and utilizing the same reasoning as for determining $M_1(z)$ gives the piecewise classical solution
\begin{equation}
\mu_i^T(z) = e_i^T(C_2 - Q_0K)\Psi_{i+p}(z,0)
+ \textstyle\int_0^z\frac{1}{\lambda_{i+p}(\zeta)}e_i^T(A_2(\zeta)K - E_2^TG(\zeta))\Psi_{i+p}(z,\zeta)\d\zeta
\end{equation}
of \eqref{ivpN2} for $i = 1,2,\ldots,m$. From these results, the piecewise classical solution of \eqref{decouplODE} results as
\begin{equation}
N_I(z) = E_1M_1(z) + E_2M_2(z).
\end{equation}

\subsection{Solution of the BVP for $P_I(z,\zeta)$}\label{sec:P_isolve}
In what follows, it is shown that the special form
\begin{equation}\label{Pdef}
P_I(z,\zeta) = \begin{bmatrix}
P_{I,1}(z,\zeta) & 0 \\
0 & 0
\end{bmatrix}              
\end{equation}
of the kernel $P_I(z,\zeta)$ with the upper triangular matrix
\begin{align}\label{P1def}
P_{I,1}(z,\zeta) &= E_1^TP_I(z,\zeta)E_1\nonumber\\
&=
\begin{bmatrix} p_{I,11}(z,\zeta)        & \ldots              &  \ldots  & p_{I,1p}(z,\zeta)\\
0                        & \ddots              & \ddots     & \vdots\\
\vdots                   & \ddots              &  \ddots    & \vdots\\
0                        & \ldots              &   0 & p_{I,pp}(z,\zeta)
\end{bmatrix} \in \mathbb{R}^{p \times p}
\end{align}
is sufficient to satisfy the BVP \eqref{decouplBVP}. To this end, premultiply the BC \eqref{BCcomp} by $E_1E_1^T + E_2E_2^T = I_n$ and equate coefficients w.r.t. $E_1$ and $E_2$ giving
\begin{equation}\label{bcsplit1}
E_1^TP_I(z,0)\Lambda(0)(E_1 + E_2Q_0) = E_1^T(N_I(z)B - A_0(z) + \mathcal{T}^{-1}_2[H_0](z))
\end{equation}
and
\begin{equation}\label{bcsplit2}
E_2^TP_I(z,0)\Lambda(0)(E_1 + E_2Q_0) = E_2^T(N_I(z)B - A_0(z) + \mathcal{T}^{-1}_2[H_0](z)). 
\end{equation}
Then, solving \eqref{bcsplit2} for $E_2^TH_0(z) = H_2(z)$ yields
\begin{equation}\label{bch2}
H_2(z) = -M_2(z)B + A_2(z)
\end{equation}
in view of \eqref{invT2}, \eqref{Nidef}, \eqref{Pdef} and $A_2(z) = E_2^TA_0(z)$. Let $[M]_{i \geq j}$ and  $[M]_{i \leq j}$ be the elements $m_{ij}$ of a matrix $M = [m_{ij}]$, that satisfy $i \geq j$ and $i \leq j$, respectively. This means that the \emph{lower} and the \emph{upper triangular part} of $M$ is extracted. Similarly, $[M]_{i > j}$ and $[M]_{i < j}$ are the elements $m_{ij}$ with $i > j$ and $i < j$, respectively. Then, applying $[\,\cdot\,]_{i > j}$ to \eqref{bcsplit1} and using \eqref{invT2}, \eqref{Nidef}, \eqref{Pdef}, $A_1(z) = E_1^TA_0(z)$ and $H_1(z) = E_1^TH_0(z)$ results in
\begin{equation}
[H_1(z)]_{i > j}\label{PBCresh}
 = [-M_1(z)B - \textstyle\int_0^zP_{I,1}(z,\zeta)H_1(\zeta)\d\zeta + A_1(z)]_{i > j},
\end{equation}
where $[E_1^TP_{I}(z,0)\Lambda(0)(E_1 + E_2Q_0)]_{i > j} = 0$ follows from \eqref{Pdef} and \eqref{P1def}. This determines the elements of the strictly lower triangular submatrix in $H_1(z)$ (see \eqref{h01def}). Finally, with \eqref{Nidef}, $\Lambda(0)E_i = E_i\Lambda_i(0)$, $i = 1,2$, and $\Lambda_i(0) = E_i^T\Lambda(0)E_i$ taking into account the conditions of the remaining elements in \eqref{bcsplit1} are
\begin{equation}\label{PBCres}
[P_{I,1}(z,0)\Lambda_1(0)]_{i \leq j} = [M_1(z)B + \textstyle\int_0^zP_{I,1}(z,\zeta)H_1(\zeta)\d\zeta]_{i \leq j}.
\end{equation}
Therein, $[H_1(z)]_{i \leq j} = 0$ and $[A_1(z)]_{i \leq j} = 0$ were used, which is implied by \eqref{a01def} and \eqref{h01def}. By inserting \eqref{Pdef} in \eqref{comu} the additional BC
\begin{equation}\label{combedp1}
[P_{I,1}(z,z)]_{i < j} = 0
\end{equation}
is obtained for $p > 1$. Hence, the BC to be fulfilled by $P_{I,1}(z,\zeta)$ are \eqref{PBCres} and \eqref{combedp1}. Since this is only a requirement for the upper triangular part of $P_{I,1}(z,\zeta)$ and the right hand side of the kernel PDE \eqref{couplpde} is equal to zero,  it is sufficient to consider $P_{I,1}(z,\zeta)$ in the special form \eqref{P1def}.

In order to determine the solution of the PDE \eqref{couplpde} define
\begin{equation}\label{sigmadef}
\sigma_{ij}(z,\zeta) = \phi_i^{-1}(\phi_i(z) - \phi_j(\zeta))
\end{equation}
for $i,j = 1,2,\ldots,p$, $i \leq j$, in view of \eqref{phidef}. Note that $\phi_i^{-1}$, $i = 1,2,\ldots,p$, exists, because $\phi_i$ in \eqref{phidef} is strictly monotonically increasing. Furthermore, the property
\begin{equation}\label{sigma0}
\sigma_{ij}(z,0) = z
\end{equation}
implied by \eqref{phidef} and \eqref{sigmadef} is utilized throughout the following. Then, by making use of the \emph{method of characteristics} the (generalized) solution \eqref{Pdef} with \eqref{P1def} of the PDE \eqref{couplpde} satisfying \eqref{combedp1} is given by
\begin{subequations}\label{solp2}
	\begin{align}
	p_{I,ii}(z,\zeta) &= \textstyle\frac{1}{\lambda_i(\zeta)}f_{ii}(\sigma_{ii}(z,\zeta))\label{diagres}\\
	p_{I,ij}(z,\zeta) &= \begin{cases}
	0, & \phi_j(\zeta) > \phi_i(z)\\
	\textstyle\frac{1}{\lambda_j(\zeta)}f_{ij}(\sigma_{ij}(z,\zeta)), & \phi_j(\zeta) \leq \phi_i(z)
	\end{cases}\label{diagres2}
	\end{align}	
\end{subequations}
for $i < j$ and $i,j = 1,2,\ldots,p$ and some at least piecewise continuous functions $f_{ij}$.

As will be shown in the sequel, the condition \eqref{PBCresh} leads to Volterra integral equations for the elements $h_{ij}(z)$, $i > j$, of $H_1(z)$ in \eqref{h01def}. Thereby, the number of integral equations increases with $p$. Therefore, at first the cases $p = 1$ and $p = 2$ are considered, as they frequently appear in applications. Subsequently, the solution of the general case $p \in \mathbb{N}$ is presented.

\subsubsection{Systems of $m+1$ coupled transport PDEs ($p = 1$)}\label{sec:m+1}
In this case $A_1(z) = 0$ and $H_1(z) = 0$ has to be taken into account. Furthermore,  let $B = b \in \mathbb{R}^{n_{\xi}}$. The only BC to be fulfilled by \eqref{solp2} is
\begin{equation}
p_{I,11}(z,0)\lambda_1(0) = M_1(z)b
\end{equation}
in view of \eqref{Nidef}, \eqref{PBCres} and $(M_1(z))^T \in \mathbb{R}^{n_{\xi}}$.
Hence, with \eqref{sigma0} and \eqref{diagres} the solution $P_I(z,\zeta)$ of the BVP \eqref{decouplBVP} is \eqref{Pdef} and
\begin{align}\label{Psol}
P_{I,1}(z,\zeta) = p_{I,11}(z,\zeta) = \textstyle\frac{1}{\lambda_1(\zeta)}M_1(\sigma_{11}(z,\zeta))b.
\end{align}
In order to determine $P(1,\zeta)$ (see \eqref{gains}) insert \eqref{Pdef} and \eqref{Psol} in \eqref{recrelp1z} to obtain the \emph{Volterra integral equation of the second kind}
\begin{equation}\label{VIGL}
p_{11}(1,\zeta) + \textstyle\int_{\zeta}^{1}p_{11}(1,\zeta')p_{I,11}(\zeta',\zeta)\d\zeta' = p_{I,11}(1,\zeta)
\end{equation}
for $p_{11}(1,\zeta)$ in
\begin{equation}\label{Porigsol}
P(1,\zeta) = \begin{bmatrix}
p_{11}(1,\zeta) & 0^T \\
0 & 0
\end{bmatrix}.              
\end{equation}
Thereby, the particular form \eqref{Porigsol} of $P(1,\zeta)$ can be easily inferred from \eqref{recrelp1z}. As the kernel $p_{I,11}(\zeta',\zeta)$ and $p_{I,11}(1,\zeta)$ in \eqref{VIGL} are continuous functions, there exists a unique continuous solution $p_{11}(1,\zeta)$ of \eqref{VIGL} (see, e.\:g., \cite[Th. 3.1]{Lin85}). The latter can be readily obtained from a \emph{successive approximation}. With this and the result of Section \ref{sec:N_isolve} the feedback gains follow from \eqref{gains}.

\subsubsection{Systems of $m+2$ coupled transport PDEs ($p = 2$)}\label{sec:p2}
In view \eqref{a01def} and \eqref{h01def} one obtains from \eqref{PBCresh} the result
\begin{equation}\label{h21}
h_{21}(z) + \textstyle\int_{0}^{z}p_{I,22}(z,\zeta)h_{21}(\zeta)\d\zeta = -[M_1(z)B]_{21} + a_{21}(z)
\end{equation}
(for the related notation see Section \ref{sec:P_isolve}). Moreover, evaluating \eqref{PBCres} gives
\begin{subequations}
	\begin{align}
	p_{I,i2}(z,0)\lambda_2(0) &= [M_1(z)B]_{i2}, \quad i = 1,2\label{rbsp}\\
	p_{I,11}(z,0)\lambda_1(0)  &= [M_1(z)B]_{11} + \textstyle\int_0^zp_{I,12}(z,\zeta)h_{21}(\zeta)\d\zeta.\label{intrb}
	\end{align}
\end{subequations}
By utilizing \eqref{sigma0}, \eqref{solp2} and \eqref{rbsp} it follows that
\begin{subequations}\label{solp2p}
	\begin{align}
	p_{I,12}(z,\zeta) &= \begin{cases}
	0, & \phi_2(\zeta) > \phi_1(z)\\
	\textstyle\frac{1}{\lambda_2(\zeta)}[M_1(\sigma_{12}(z,\zeta))B]_{12}, & \phi_2(\zeta) \leq \phi_1(z)
	\end{cases}\label{diagres2p}\\
	p_{I,22}(z,\zeta) &= \textstyle\frac{1}{\lambda_2(\zeta)}[M_1(\sigma_{22}(z,\zeta))B]_{22}.\label{diagresp}
	\end{align}	
\end{subequations}
With this, \eqref{h21} becomes a Volterra integral equation of the second kind for $h_{21}(z)$. Since $p_{I,11}(z,\zeta)$ is continuous and $a_{21}(z)$ is piecewise continuous (see Appendix \ref{appA}) the integral equation \eqref{h21} has a unique piecewise continuous solution (see, e.\:g., \cite[Th. 3.2]{Lin85}). Then, by utilizing the resulting $h_{21}(z)$ in \eqref{intrb} the BC for $p_{I,11}(z,\zeta)$ is known. Hence, \eqref{diagres} yields
\begin{equation}
p_{I,11}(z,\zeta)  = \textstyle\frac{1}{\lambda_1(\zeta)}([M_1(\sigma_{11}(z,\zeta))B]_{11}
+ \textstyle\int_0^{\sigma_{11}(z,\zeta)}p_{I,12}(\sigma_{11}(z,\zeta),\zeta')h_{21}(\zeta')\d\zeta').
\end{equation}
Finally, with $P_{I}(z,\zeta)$ being determined the matrix
\begin{equation}\label{P1def0}
P(1,\zeta) = \begin{bmatrix}
P_{1}(1,\zeta) &  0\\
0               & 0
\end{bmatrix}
\end{equation}
with
\begin{equation}
P_1(1,\zeta) = \begin{bmatrix}
p_{11}(1,\zeta) &  p_{12}(1,\zeta)\\
0               &  p_{22}(1,\zeta) 
\end{bmatrix}
\end{equation}
follows from solving \eqref{recrelp1z}, which completes the state feedback design.

\subsubsection{General heterodirectional systems ($p \in \mathbb{N}$)}
In what follows the approach of the last paragraph is extended to the general case $p \in \mathbb{N}$. From \eqref{PBCres} one obtains the result
\begin{equation}
p_{I,lp}(z,0)\lambda_{p}(0) = [M_1(z)B]_{lp}, \quad l = 1,2,\ldots,p,
\end{equation}
because the last column of $P_{I,1}(z,\zeta)H_1(\zeta)$ is vanishing. In order to determine the remaining $p_{I,lj}(z,0)$, $i = 1,2,\ldots,p$, $j = 1,2,\ldots,p-1$, observe that the elements of the integrands in \eqref{PBCresh} and \eqref{PBCres} can be represented by
\begin{equation}\label{sumrule}
[P_{I,1}(z,\zeta)H_1(\zeta)]_{ij} = \!\sum_{k = \max(i,j+1)}^{p}p_{I,ik}(z,\zeta)h_{kj}(\zeta)
\end{equation}
for $i = 1,2,\ldots,p$, $j = 1,2,\ldots,p-1$. This is implied by the upper and lower triangular form of $P_{I,1}(z,\zeta)$ and $H_1(\zeta)$ (see \eqref{P1def} and \eqref{h01def}).  With these BC the kernel elements $p_{I,lp}(z,\zeta)$, $l = 1,2,\ldots,p$, follow from \eqref{solp2}. In the sequel the remaining elements  $p_{I,lj}(z,\zeta)$, $j = p-1,p-2,\ldots,l$, in the $l$-th line of $P_{I,1}(z,\zeta)$ are determined (see the triangular form in \eqref{P1def}). Thereby, it is assumed that $l$ is decreasing starting from $p-1$, i.\:e., $l = p-1,p-2,\ldots,1$. Consider the BC \eqref{PBCres} and apply \eqref{sumrule} with $i = l \leq j$ to get
\begin{equation}
p_{I,lj}(z,0)\lambda_j(0) = [M_1(z)B]_{lj} + \sum_{k = j+1}^{p}\textstyle\int_0^zp_{I,lk}(z,\zeta)h_{kj}(\zeta)\d\zeta,
\end{equation}
which is successively evaluated for $j = p-1,p-2,\ldots,l$. Note that the function $h_{kj}(\zeta)$, $k = j+1,\ldots,p$, have been already computed and thus are known. With this, the BC $p_{I,lj}(z,0)$ can be  obtained so that the remaining elements $p_{I,lj}(z,\zeta)$, $j =  l,l+1, \ldots, p-2,p-1$, are determined by \eqref{solp2}. By making use of this result the BC \eqref{PBCresh} and \eqref{sumrule} with $i = l > j$ leads to the Volterra integral equations
\begin{equation}\label{hvolp}
h_{lj}(z) + {\textstyle\int_0^z}p_{I,ll}(z,\zeta)h_{lj}(\zeta)\d\zeta
 = - \sum\limits_{k = l+1}^{p}\textstyle\int_0^zp_{I,lk}(z,\zeta)h_{kj}(\zeta)\d\zeta -[M_1(z)B]_{lj} + a_{lj}(z)
\end{equation}
for $h_{lj}(z)$ in the $l$-th line of $H_1(\zeta)$ with $j = 1,2,\ldots,l-1$ (see the triangular form in \eqref{h01def}). Therein, the integrand  $p_{I,lk}(z,\zeta)h_{kj}(\zeta)$, $k = l+1,\ldots,p$, $j = 1,2,\ldots,l-1$, appearing in the second line of \eqref{hvolp} has been computed in the previous steps. Hence, the right hand side of \eqref{hvolp} is known. By utilizing this procedure line by line the matrix $P_{I,1}(z,\zeta)$ and thus \eqref{Pdef} can be found. Then, the feedback gains \eqref{gains} follow from solving \eqref{recrelp1z} by making use of \eqref{P1def0} and utilizing a matrix $P_{1}(z,\zeta)$ with the same upper triangular form as in \eqref{P1def}. 
The preceding derivations show that $p^2$ Volterra integral equations have to be solved for determining the inverse decoupling transformation \eqref{ccord2}.

The next theorem summarizes the results of this section.
\begin{theorem}[Inverse decoupling equations]
	The inverse decoupling equations \eqref	{decouplODE} and \eqref{decouplBVP} have a solution for $p \in \mathbb{N}$ with the elements of $N_I(z)$ and  $P_I(z,\zeta)$ being piecewise $C^1$.	 
\end{theorem}

\section{Observer design}\label{sec:obs}
Consider the \emph{observer}
\begin{subequations}\label{obs}
	\begin{align}
	\partial_t\hat{x}(z,t) &= \Lambda(z)\partial_z\hat{x}(z,t) + A(z)\hat{x}(z,t) + C_1(z)\hat{\xi}(t) + L(z)(y(t) - \hat{x}_1(0,t))\label{xeqo}\\
	\hat{x}_2(0,t) &= Q_0y(t) + C_2\hat{\xi}(t)\label{uabco}\\
	\hat{x}_1(1,t) &= Q_1\hat{x}_2(1,t) + u(t)\label{aktbco}\\
	\dot{\hat{\xi}}(t) &= F\hat{\xi}(t) + By(t) + L_{\xi}(y(t) - \hat{x}_1(0,t))\label{plantodeo}
	\end{align}
\end{subequations}
for \eqref{plant} with \eqref{xeqo} defined on $(z,t) \in (0,1) \times \mathbb{R}^+$ and \eqref{uabco}--\eqref{plantodeo} on $t \in \mathbb{R}^+$. Thereby, $L \in (L_2(0,1))^{n \times p}$ and $L_{\xi} \in \mathbb{R}^{n_{\xi}\times p}$ are the \emph{observer gains} as well as $\hat{x}(z,0) = \hat{x}_0(z) \in \mathbb{R}^{n}$, $z \in [0,1]$, and $\hat{\xi}(0) = \hat{\xi}_0 \in \mathbb{R}^{n_{\xi}}$ are the observer IC. 

Introduce the \emph{error states} $\varepsilon_x = x - \hat{x}$ with $\varepsilon_{x_i} = E_i^T\varepsilon$, $i = 1,2$, and $\varepsilon_{\xi} = \xi - \hat{\xi}$ to obtain the \emph{observer error dynamics}
\begin{subequations}\label{obse}
	\begin{align}
	\partial_t\varepsilon_x(z,t) &= \Lambda(z)\partial_z\varepsilon_x(z,t) + A(z)\varepsilon_x(z,t) + C_1(z)\varepsilon_{\xi}(t) - L(z)\varepsilon_{x_1}(0,t)\label{xeqoe}\\
	\varepsilon_{x_2}(0,t) &= C_2\varepsilon_{\xi}(t)\label{uabcoe}\\
	\varepsilon_{x_1}(1,t) &= Q_1\varepsilon_{x_2}(1,t)\label{aktbcoe}\\
	\dot{\varepsilon}_{\xi}(t) &= F\varepsilon_{\xi}(t) - L_{\xi}\varepsilon_{x_1}(0,t).\label{plantodeoe}
	\end{align}
\end{subequations}
The corresponding design procedure follows similar lines as the computation of the state feedback controller in Section \ref{sec:statefeed}, which is inspired by the corresponding results in \cite{Deu16c}. In the first step a backstepping transformation maps the PDE subsystem \eqref{xeqoe}--\eqref{aktbcoe} into a target system of simpler structure. Then, the \emph{decoupling equations} to be solved for obtaining a \emph{PDE-ODE cascade} in the second step become explicitly solvable. 

Since the resulting observer gains depend on the inverse transformation, the invertible \emph{backstepping coordinates}
\begin{equation}\label{obstrafo}
\varepsilon_x(z,t) = \tilde{\varepsilon}_x(z,t) - \textstyle\int_0^zR_I(z,\zeta)\tilde{\varepsilon}_x(\zeta,t)\d\zeta
= \mathcal{T}_o^{-1}[\tilde{\varepsilon}_x(t)](z)
\end{equation}
are introduced for \eqref{xeqoe}--\eqref{aktbcoe}. The \emph{integral kernel} $R_I(z,\zeta) \in \mathbb{R}^{n \times n}$ in \eqref{obstrafo} is the solution of the \emph{kernel equations}
\begin{subequations}\label{obsbvp}
	\begin{align}
	\Lambda(z)\partial_zR_I(z,\zeta) + \partial_{\zeta}(R_I(z,\zeta)\Lambda(\zeta)) &= -A(z)R_I(z,\zeta), \quad 0 < \zeta < z < 1\label{ocdbvp1}\\
	(E_1^T - Q_1E_2^T)R_I(1,\zeta) &= -S(\zeta)\label{obsbvpbc}\\
	\Lambda(z) R_I(z,z) - R_I(z,z)\Lambda(z) &= A(z).\label{ocdbvp2}
	\end{align}
\end{subequations}
In \eqref{obsbvpbc} the matrix $S(\zeta) =\begin{bmatrix}S_{1}(\zeta) & S_{2}(\zeta)\end{bmatrix}$ consists of the strictly upper triangular matrix 
\begin{equation}\label{h01defs}
S_{1}(\zeta) = \begin{bmatrix} 0    & s_{12}(\zeta) & \ldots  & s_{1p}(\zeta)\\
\vdots  & \ddots        & \ddots  & \vdots\\
\vdots  & \ddots        & \ddots  & s_{p-1 \,p}(\zeta)\\
0       & \ldots        & \ldots  & 0
\end{bmatrix} \in \mathbb{R}^{p \times p}
\end{equation}
and of $S_{2}(\zeta) \in \mathbb{R}^{p \times m}$, which has no special form. In Appendix \ref{appB} it is shown that \eqref{obsbvp} can be traced back to kernel equations analysed in \cite{Hu15b}. From this, the existence of a unique piecewise $C^1$-solution $R_I(z,\zeta)$ w.r.t. \eqref{obsbvp} as well as the existence of the corresponding Volterra-type integral transformation $\mathcal{T}_o[\cdot]$ can be inferred (see \eqref{obstrafo}). In particular,
\begin{equation}\label{invtrafo}
\tilde{\varepsilon}_x(z,t) = \varepsilon_x(z,t) + \textstyle\int_0^zR(z,\zeta)\varepsilon_x(\zeta,t)\d\zeta
= \mathcal{T}_o[\varepsilon_x(t)](z)
\end{equation}
holds with the integral kernel $R(z,\zeta) \in \mathbb{R}^{n \times n}$. 
By inserting \eqref{obstrafo} in \eqref{invtrafo} and changing the order of integration the \emph{reciprocity relation}
\begin{equation}\label{recrelcontrolobs}
R(z,\zeta) - \textstyle\int_{\zeta}^{z}R(z,\zeta')R_I(\zeta',\zeta)\d\zeta' = R_I(z,\zeta)
\end{equation}
of the kernels $R(z,\zeta)$ and $R_I(z,\zeta)$ is readily obtained. It can be utilized to obtain the kernel $R(z,\zeta)$ from a known $R_I(z,\zeta)$. Furthermore, the result
\begin{equation}\label{kerntrafo}
 \mathcal{T}_o[R_I(\cdot,0)] = R_I(z,0) + \textstyle\int_{0}^{z}R(z,\zeta')R_I(\zeta',0)\d\zeta' = R(z,0)
\end{equation}
follows by setting $\zeta = 0$ in \eqref{recrelcontrolobs}. After solving \eqref{obsbvp} the elements of the strictly upper triangular part in \eqref{h01defs} and the elements of $S_2(\zeta)$ are determined by the resulting kernel. 

Utilizing the usual calculations, $\tilde{\varepsilon}_{x_i} = E_i^T\tilde{\varepsilon}_x$, $i = 1,2$, and \eqref{kerntrafo} it is readily found that \eqref{obstrafo} transforms the error system \eqref{obse} into
\begin{subequations}\label{obserrsys2}
	\begin{align}
	\partial_t\tilde{\varepsilon}_x(z,t) &= \Lambda(z)\partial_z\tilde{\varepsilon}_x(z,t)  + G_o(z)\varepsilon_{\xi}(t) -\widetilde{L}(z)\tilde{\varepsilon}_{x_1}(0,t)\label{opdes1}\\	                       
	\tilde{\varepsilon}_{x_2}(0,t) &= C_2\varepsilon_{\xi}(t)\label{obsrb2e2}\\
	\tilde{\varepsilon}_{x_1}(1,t) &= Q_1\tilde{e}_{x_2}(1,t) - \textstyle\int_0^1S(\zeta)\tilde{\varepsilon}_x(\zeta,t)\d\zeta\label{dobsue2}\\
	\dot{\varepsilon}_{\xi}(t) &= F\varepsilon_{\xi}(t) - L_{\xi}\tilde{\varepsilon}_{x_1}(0,t). \label{obsode2}
	\end{align}
\end{subequations}
Therein, $G_o(z) = \mathcal{T}_o[C_1](z) - R(z,0)\Lambda(0)E_2C_2$ and the \emph{auxiliary observer gain}
\begin{equation}\label{obsgain}
\widetilde{L}(z) = \mathcal{T}_o[L](z) + \mathcal{T}_o[R_I(\cdot,0)](z)\Lambda(0)E_1
\end{equation}
were used. In the second step the PDE target system \eqref{opdes1}--\eqref{dobsue2} is decoupled from the ODE subsystem \eqref{obsode2} by introducing the new coordinates
\begin{equation}\label{newcoordODE}
\vartheta(z,t) = \tilde{\varepsilon}_x(z,t) - \Gamma(z)\varepsilon_{\xi}(t)
\end{equation}
with $\Gamma(z) \in \mathbb{R}^{n \times n_{\xi}}$. Take the time derivative of \eqref{newcoordODE}, use $\vartheta_{i} = E_i^T\vartheta$, $i = 1,2$, and insert \eqref{obserrsys2}. Then, the \emph{ODE-PDE cascade}
\begin{subequations}\label{obserrsys3}
	\begin{align}
	\partial_t\vartheta(z,t) &= \Lambda(z)\partial_z\vartheta(z,t),&& (z,t) \in (0,1)\times\mathbb{R}^+ \label{opdes13}\\	                       
	\vartheta_{2}(0,t) &= 0, && t > 0\label{dobsue23}\\
	\vartheta_{1}(1,t) &= Q_1\vartheta_{2}(1,t) - \textstyle\int_0^1S(\zeta)\vartheta(\zeta,t)\d\zeta, && t > 0\label{obsrb2e23}\\
	\dot{\varepsilon}_{\xi}(t) &= (F - L_{\xi}E_1^{T}\Gamma(0))\varepsilon_{\xi}(t) - L_{\xi}\vartheta_1(0,t), && t > 0\label{odetildobsfin}
	\end{align}
\end{subequations}
results, if $\Gamma(z)$ is the solution of the \emph{decoupling equations}
\begin{subequations}\label{odepdebvp}
	\begin{align}
	\Lambda(z) \d_z\Gamma(z) &= \Gamma(z)F - G_o(z),\quad z \in (0,1)\label{pdeobsdecpl}\\
	E_2^T\Gamma(0) &= C_2\\
	(E_1^T - Q_1E_2^T)\Gamma(1) &=  -\textstyle\int_0^1S(\zeta)\Gamma(\zeta)\d\zeta
	\end{align}
\end{subequations}
and 
\begin{equation}\label{Ltilcomp}
 \widetilde{L}(z) = \Gamma(z)L_{\xi}
\end{equation}
holds. The next theorem clarifies the solvability of \eqref{odepdebvp}.
\begin{theorem}[ODE-PDE cascade]\label{lem:odepdecasc}
The BVP \eqref{odepdebvp} has a unique solution $\Gamma(z) = [\Gamma_{ij}(z)] \in \mathbb{R}^{n \times n_{\xi}}$ with $\Gamma_{ij}$ piecewise $C^1$-functions, $i = 1,2,\ldots,n$, $j = 1,2,\ldots,n_{\xi}$.
\end{theorem}
\begin{proof}
With $E_1E_1^T + E_2E_2^T = I_n$ (see \eqref{Edef}) and making use of the same approach as in Section \ref{sec:N_isolve} the BVP \eqref{odepdebvp} can be written as
\begin{subequations}\label{ivpN1e}
	\begin{align}
	\d_zE_1^T\Gamma(z)  &= \Lambda^{-1}_1(z)(E_1^T\Gamma(z)F - E_1^TG_o(z)), \quad z \in (0,1)\label{N1odee}\\
	E_1^T\Gamma(1) &= Q_1E_2^T\Gamma(1) - \textstyle\int_0^1S(\zeta)\Gamma(\zeta)\d\zeta
	\end{align}
\end{subequations}
and 
\begin{subequations}\label{ivpN2e}
	\begin{align}
	\d_zE_2^T\Gamma(z) &= \Lambda^{-1}_2(z)(E_2^T\Gamma(z)F - E_2^TG_o(z)), \quad z \in (0,1)\label{N2odee}\\
	E_2^T\Gamma(0) &= C_2.\label{N2bce}
	\end{align}
\end{subequations}
By introducing $\mu_i^T(z) = e_i^TE_2^T\Gamma(z)$, $i = 1,2,\ldots,m$, one obtains with $g_i^T(z) = -\frac{1}{\lambda_{i+p}(z)}e_i^TE_2^TG_o(z)$ and $c_i^T = e_i^TC$ from \eqref{ivpN2e}
\begin{subequations}\label{ivpN2setme}
	\begin{align}
	\d_z\mu_i^T(z) &= \textstyle\frac{1}{\lambda_{i+p}(z)}\mu_i^T(z)F\label{N2odeme}
	+ g_i^T(z), \quad z \in (0,1)\\
	\mu_i^T(0) &= c_i^T.\label{N2bcme}
	\end{align}
\end{subequations}
Observe that 
\begin{equation}\label{fmatdefe}
\Psi_i(z,\zeta) = \text{e}^{F(\phi_i(z) - \phi_i(\zeta))}, \quad i = 1,2,\ldots,n
\end{equation}
with $\phi_i$ defined in \eqref{phidef} is the solution of the IVP
\begin{equation}\label{psiivp}
\partial_z\Psi_i(z,\zeta) = \textstyle\frac{1}{\lambda_i(z)}F\Psi_i(z,\zeta), \quad \Psi_i(\zeta,\zeta) = I.
\end{equation}
Hence, $\Psi_{i+p}(z,\zeta)$, $i = 1,2,\ldots,m$, is the \emph{fundamental matrix} related to IVP \eqref{ivpN2setme} so that the corresponding solution reads
\begin{equation}\label{musolobs}
\mu_i^T(z) = c_i^T\Psi_{i+p}(z,0)
+ \textstyle\int_0^zg_i^T(\zeta)\Psi_{i+p}(z,\zeta)\d\zeta.
\end{equation}
Correspondingly, with $\nu_i^T(z) = e_i^TE_1^T\Gamma(z)$, $i = 1,2,\ldots,p$, $h_i^T(z) = -\frac{1}{\lambda_{i}(z)}e_i^TE_1^TG_o(z)$ and $q_i^T = e_i^TQ_1E_2^T\Gamma(1))$ the BVP \eqref{ivpN1e} becomes
\begin{subequations}\label{ivpN2setmee}
	\begin{align}
	\d_z\nu_i^T(z) &= \textstyle\frac{1}{\lambda_i(z)}\nu_i^T(z)F\label{N2odemee}
	+ h_i^T(z), \quad z \in (0,1)\\
	\nu_i^T(1) &= q_i^T - \textstyle\int_0^1e_i^TS(\zeta)(E_1E_1^T\Gamma(\zeta) + E_2E_2^T\Gamma(\zeta))\d\zeta.\label{N2bcmee}
	\end{align}
\end{subequations}
By inserting the solution 
\begin{equation}\label{solcomp}
 E_1^T\Gamma(z) = N_0(z) + H(z)
\end{equation}
of \eqref{N2odemee} in \eqref{N2bcmee}, where
\begin{equation}
N_0(z) = \begin{bmatrix}
\nu_1^T(0)\Psi_1(z,0)\\
\vdots\\
\nu_p^T(0)\Psi_p(z,0)
\end{bmatrix} \quad \text{and} \quad
H(z) =  \int_0^z\begin{bmatrix}h_1^T(\zeta)\Psi_1(z,\zeta)\\
\vdots\\
h_p^T(\zeta)\Psi_p(z,\zeta)
\end{bmatrix}\d\zeta
\end{equation}
as well as taking the partition of $S(\zeta)$ into account the result
\begin{equation}\label{inidet}
 \nu_i^T(0)\Psi_i(1,0) + \textstyle\int_0^1e_i^TS_1(\zeta)N_0(\zeta)\d\zeta = r_i^T, \quad i = 1,2,\ldots,p
\end{equation}
with
\begin{equation}
 r_i^T = q_i^T - \textstyle\int_0^1(e_i^T(S_1(\zeta)H(\zeta) + S_2(\zeta)E_2^T\Gamma(\zeta))- h_i^T(\zeta)\Psi_{i}(1,\zeta))\d\zeta
\end{equation}
follows. In view of \eqref{h01defs} the relation \eqref{inidet} can be rewritten as
\begin{multline}\label{prefindecoul}
 \begin{bmatrix}
  \nu_1^T(0)\Psi_1(1,0)\\
  \vdots\\
  \nu_{p-1}^T(0)\Psi_{p-1}(1,0)\\
  \nu_p^T(0)\Psi_p(1,0)
 \end{bmatrix}
 +
 \int_0^1\begin{bmatrix}
  \nu_2^T(0)\Psi_2(\zeta,0)s_{12}(\zeta) + \ldots + \nu_p^T(0)\Psi_p(\zeta,0)s_{1p}(\zeta)\\
  \vdots\\
  \nu_p^T(0)\Psi_p(\zeta,0)s_{p-1\,p}(\zeta)\\
  0^T
 \end{bmatrix}\d\zeta\\
  = \begin{bmatrix}
  	 r_1^T\\
  	 \vdots\\
  	 r_p^T
  \end{bmatrix}.
\end{multline}
By introducing
\begin{equation}
 M_{ij} = \textstyle\int_0^1\Psi_j(\zeta,0)s_{ij}(\zeta)\d\zeta
\end{equation}
one obtains 
\begin{equation}\label{nu0solv}
 \begin{bmatrix}
  \nu_1^T(0) & \ldots & \nu_p^T(0)
 \end{bmatrix}M = \begin{bmatrix}
 r_1^T & \ldots & r_p^T
 \end{bmatrix}
\end{equation}
from \eqref{prefindecoul}, in which
\begin{equation}\label{Mdef}
 M = \begin{bmatrix}
  \Psi_1(1,0) & 0      & \ldots            & 0              & 0\\
  M_{12}      & \ddots & \ddots            & \vdots         & \vdots\\
  \vdots      & \ddots & \ddots            & 0              & 0\\
  M_{1\,p-1}  & \ldots &  M_{p-2\,p-1}     &\Psi_{p-1}(1,0) & 0\\
  M_{1p}	  & \ldots & M_{p-2\,p}        &M_{p-1\,p}      & \Psi_p(1,0)
 \end{bmatrix}
\end{equation}
holds. Hence,
\begin{equation}
 \det M = \prod_{i=1}^p\det\Psi_i(1,0) = \prod_{i=1}^p\det \text{e}^{F\phi_i(1)} 
        = \prod_{i=1}^p \text{e}^{\operatorname{tr}(F\phi_i(1))} \neq 0
\end{equation}
follows from \eqref{fmatdefe} and \eqref{Mdef}. Consequently, the initial values $\nu_i^T(0)$, $i = 1,2,\ldots,p$, determining the solution \eqref{solcomp}  of \eqref{ivpN2setmee} can be calculated uniquely from \eqref{nu0solv}. Together with \eqref{musolobs} this shows that there exists a unique solution of  \eqref{odepdebvp}, that is piecewise $C^1$ as $G_o(z)$ in \eqref{pdeobsdecpl} has this property (see the definition of $G_o(z)$ and the properties of $R(z,\zeta)$ derived in \cite{Deu16c}). 
\end{proof}

In \cite{Hu15b} it is shown that the infinite-dimensional subsystem \eqref{opdes13}--\eqref{obsrb2e23} is \emph{finite-time stable} for piecewise continuous IC, i.\:e.,
\begin{equation}\label{eq:to}
  \tilde{\varepsilon}_x(z,t) = 0, \quad t \geq t_o = \sum_{i=1}^{p+1}|\phi_i(1)|.
\end{equation}
Hence, if the ODE-subsystem \eqref{odetildobsfin} is asymptotically stable so is the ODE-PDE cascade \eqref{obserrsys3}. For this, the observer gain $L_{\xi}$ has to be such that  $F - L_{\xi}E_1^{T}\Gamma(0)$ is a Hurwitz matrix. The next theorem presents the corresponding observability condition for $(E_1^{T}\Gamma(0),F)$.
\begin{theorem}[Observability]\label{lem:obs}
Let $N_{\xi}(s)$ be the numerator of the transfer matrix $F_{\xi}(s)= D_{\xi}^{-1}(s)N_{\xi}(s)$ w.r.t. \eqref{xeq}--\eqref{aktbc} and \eqref{meas} from $\xi$ to $y$ and denote the linearly independent eigenvectors of $F$  w.r.t. the eigenvalues $\mu_{i}$, $i = 1,2,\ldots,\bar{n}_{\xi}$, $\bar{n}_{\xi} \leq n_{\xi}$, by $v_{i}$. Then, the pair $(E_1^T\Gamma(0),F)$ is observable iff 
	\begin{equation}\label{obsbed}
	N_{\xi}(\mu_{i})v_{i} \neq 0, \quad i = 1,2,\ldots,\bar{n}_{\xi}.
	\end{equation} 
\end{theorem}
\begin{proof}
The pair $(E_1^{T}\Gamma(0),F)$ is observable iff 
\begin{equation}\label{obsbedmod}
 E_1^T\Gamma(0)v_{i} \neq 0, \quad i = 1,2,\ldots,\bar{n}_{\xi}
\end{equation}
(see \cite[Th. 6.2.5]{KL80} and the proof of Theorem \ref{lem:odepdecasc}). 
Postmultiply \eqref{odepdebvp} by the linearly independent eigenvectors $v_i$, $i = 1,2,\ldots,\bar{n}_{\xi}$, of $F$ w.r.t. the eigenvalues $\mu_{i}$ and define $\gamma_i = \Gamma v_{i}$, $g_{i} = -\Lambda^{-1}G_ov_{i}$ and $c_i = C_2v_i$. For $i = 1,2,\ldots,\bar{n}_{\xi}$ this results in the decoupled BVPs 
\begin{subequations}
	\begin{align}
		\d_z\gamma_i(z) &= \mu_{i}\Lambda^{-1}(z)\gamma_i(z) + g_i(z), \quad z \in (0,1) \label{obsode}\\
		E_2^T\gamma_{i}(0) &= c_i\label{obsodebc1}\\
		(E_1^T - Q_1E_2^T)\gamma_{i}(1) &= -\textstyle\int_0^1S(\zeta)\gamma_{i}(\zeta)\d\zeta. \label{obsodebc2}	 
	\end{align}
\end{subequations}
The \emph{fundamental matrix} solving the IVP 
\begin{equation}
	\partial_z\Phi(z,\zeta,\mu_i) = \mu_i\Lambda^{-1}(z)\Phi(z,\zeta,\mu_i), \quad \Phi(\zeta,\zeta,\mu_i) = I, \quad i = 1,2,\ldots,\bar{n}_{\xi}
\end{equation}
reads
\begin{equation}\label{fmat}
	\Phi(z,\zeta,\mu_i) = \operatorname{diag}(\text{e}^{\mu_i(\phi_1(z) - \phi_1(\zeta))},\ldots,\text{e}^{\mu_i(\phi_n(z) - \phi_n(\zeta))}),
\end{equation}
in which $\phi_i$, $i = 1,2,\ldots,\bar{n}_{\xi}$, is defined in \eqref{phidef}. Consider 
\begin{equation}
	\gamma_i(0) = E_1E_1^T\gamma_i(0) + E_2E_2^T\gamma_i(0) = E_1E_1^T\gamma_i(0) + E_2c_i
\end{equation}	
in view of \eqref{obsodebc1}. Then, the solution of \eqref{obsode} and \eqref{obsodebc1} is
\begin{equation}
	\gamma_i(z) = \Phi(z,0,\mu_{i})E_1E_1^T\gamma_i(0) + m(z,\mu_{i})
\end{equation}
with
\begin{equation}
	m(z,\mu_{i}) = \Phi(z,0,\mu_{i})E_2c_i + \textstyle\int_0^z\Phi(z,\zeta,\mu_{i})g_{i}(\zeta)\d\zeta.
\end{equation}
By inserting this in \eqref{obsodebc2} and defining 
\begin{equation}
	M_1(\mu_{i}) = ((E_1^T - Q_1E_2^T)\Phi(1,0,\mu_{i})
	+ \textstyle\int_0^1S(\zeta)\Phi(\zeta,0,\mu_{i})\d\zeta)E_1
\end{equation}
one gets
\begin{equation}\label{mproof}
	M_1(\mu_{i})E_1^T\gamma_i(0) = M_2(\mu_i)
\end{equation} 
with 
\begin{equation}
	M_2(\mu_i) = -(E_1^T - Q_1E_2^T)m(1,\mu_i) - \textstyle\int_0^1S(\zeta)m(\zeta,\mu_i)\d\zeta
\end{equation}
after a simple calculation. Since $\Phi(z,\zeta,\mu_{i})$ is diagonal and $S(\zeta)\Phi(\zeta,0,\mu_{i})E_1$ is a strictly upper triangular (see \eqref{h01defs}) the matrix $M_1(\mu_i)$ in \eqref{mproof} is upper triangular. Hence,  $\det M_1(\mu_{i}) = \prod_{j = 1}^p\exp(\mu_{i}\phi_j(1)) \neq 0$, $i = 1,2,\ldots,n_{\xi}$, is directly obtained. With this, \eqref{mproof} can be uniquely solved for $E_1^T\gamma_i(0)$ yielding for \eqref{obsbedmod} the result
\begin{equation}
 E_1^T\Gamma(0)v_{i} = E_1^T\gamma_{i}(0) = M_1^{-1}(\mu_{i})M_2(\mu_i).
\end{equation}
Hence, \eqref{obsbedmod} holds iff $M_2(\mu_i) \neq 0$. Apply the backstepping transformation $x(z,t) = \mathcal{T}^{-1}_o[\tilde{x}(t)](z)$ (see \eqref{obstrafo}) to \eqref{xeq}--\eqref{aktbc} and \eqref{meas}. By making use of \eqref{kerntrafo} this leads to the system description
\begin{subequations}
 \begin{align}
	\partial_t\tilde{x}(z,t) &= \Lambda(z)\partial_z\tilde{x}(z,t) + A_0(z)\tilde{x}_1(0,t) - (\mathcal{T}_o[C_1](z) + R(z,0)\Lambda(0)E_2C_2)\xi(t)\label{tfode2}\\
	\tilde{x}_2(0,t) &= Q_0\tilde{x}_1(0,t) + C_2\xi(t)\\
	\tilde{x}_1(1,t) &= Q_1\tilde{x}_2(1,t) - \textstyle\int_0^1S(\zeta)\tilde{x}(\zeta,t)\d\zeta\\
	y(t) &= \tilde{x}_1(0,t).
 \end{align}
\end{subequations}
From this, the transfer matrix $F_{\xi}(s)= D_{\xi}^{-1}(s)N_{\xi}(s)$ can be calculated in closed-form. As a result one obtains the relation
\begin{equation}
 M_2(\mu_i) = N_{\xi}(\mu_i)v_i,
\end{equation}
which proves the theorem. 
\end{proof}
 \begin{rem}
The result of Theorem \ref{lem:obs} has an intuitive meaning. The condition \eqref{obsbed} ensures that the eigenmodes of the ODE subsystem \eqref{plantode} can be transferred to $y(t)$ in the steady state. Then, the observer \eqref{obs} is capable to estimate the lumped state $\xi(t)$ from the measurement $y(t)$. As a consequence, \eqref{obsbed} represents the \emph{existence condition} for the PDE-ODE system \eqref{plant} to be stabilizable by observer-based state feedback control. \hfill $\triangleleft$
 \end{rem}

After determining $L_{\xi}$ by an eigenvalue assignment for $F - L_{\xi}E_1^{T}\Gamma(0)$ (see \eqref{odetildobsfin}), the observer gain 
\begin{equation}\label{obsgainf}
L(z) = \mathcal{T}^{-1}_o[\Gamma](z)L_{\xi} - R_I(z,0)\Lambda(0) E_1
\end{equation}
can be calculated (see \eqref{obsgain} and \eqref{Ltilcomp}), which completes the design of the observer \eqref{obs}.

\section{Closed-loop stability}\label{sec:clstab}
By utilizing the state estimates $\hat{\xi}$ and $\hat{x}$ of the observer \eqref{obs} in \eqref{sfeed} one obtains an output feedback controller in form of the \emph{observer-based compensator}
\begin{subequations}\label{obsc}
	\begin{align}
	\partial_t\hat{x}(z,t) &= \Lambda(z)\partial_z\hat{x}(z,t) + A(z)\hat{x}(z,t) + C_1(z)\hat{\xi}(t) + L(z)(y(t) - \hat{x}_1(0,t))\label{xeqoc}\\
	\hat{x}_2(0,t) &= Q_0y(t) + C_2\hat{\xi}(t)\label{uabcoc}\\
	\hat{x}_1(1,t) &= Q_1\hat{x}_2(1,t) + u(t)\label{aktbcoc}\\
	\dot{\hat{\xi}}(t) &= F\hat{\xi}(t) + By(t) + L_{\xi}(y(t) - \hat{x}_1(0,t))\label{plantodeoc}\\
	u(t) &=  - Q_1\hat{x}_2(1,t) + \mathcal{K}[\hat{\xi}(t),\hat{x}(t)]\label{outc}.		
	\end{align}
\end{subequations}
In what follows the stability of the resulting closed-loop system is investigated. For this, apply \eqref{outc} to \eqref{plant} as well as use $\hat{x} = x - \varepsilon_x$ and $\hat{\xi} = x - \varepsilon_{\xi}$ (see Section \ref{sec:obs}). This results in 
\begin{subequations}\label{cplant}
 \begin{align}
	\partial_tx(z,t) &= \Lambda(z)\partial_zx(z,t) \!+\! A(z)x(z,t) \!+\! C_1(z)\xi(t), && (z,t) \in (0,1) \times \mathbb{R}^+ \label{cxeq}\\
	x_2(0,t) &= Q_0x_1(0,t) +  C_2\xi(t), &&  t > 0\label{cuabc}\\
	x_1(1,t) &= \mathcal{K}[\xi(t),x(t)] + Q_1\varepsilon_{x_2}(1,t) - \mathcal{K}[\varepsilon_{\xi}(t),\varepsilon_{x}(t)], &&  t > 0\label{caktbc}\\
	\dot{\xi}(t) &= F\xi(t) + Bx_1(0,t), && t > 0.\label{cplantode}
 \end{align}
\end{subequations}
The backstepping transformation
\begin{align}\label{btrafocl}
\tilde{x}(z,t) = x(z,t) - \textstyle\int_0^zK(z,\zeta)x(\zeta,t)\d \zeta = \mathcal{T}_1[x(t)](z)
\end{align}
(see \eqref{btrafo}) and the decoupling coordinates
\begin{equation}\label{cccord2}
\tilde{x}(z,t) =  \mathcal{T}^{-1}_2[e_x(t)](z) + N_I(z)\xi(t)
\end{equation}
(see \eqref{ccord2}) map \eqref{cplant} into
\begin{subequations}\label{cplant2}
	\begin{align}
	\partial_te_x(z,t) &= \Lambda(z)\partial_ze_x(z,t) + H_0(z)e_{x_1}(0,t), && (z,t) \in (0,1) \times \mathbb{R}^+ \label{cxeq2}\\
	e_{x_2}(0,t) &= Q_0e_{x_1}(0,t), &&  t > 0\label{cuabc2}\\
	e_{x_1}(1,t) &= Q_1\varepsilon_{x_2}(1,t) - \mathcal{K}[\varepsilon_{\xi}(t),\varepsilon_{x}(t)], &&  t > 0\label{caktbc2}\\
	\dot{\xi}(t) &= (F-BK)\xi(t) + Be_{x_1}(0,t), && t > 0,\label{cplantode2}
	\end{align}
\end{subequations}
which can easily be inferred from the results of Section \ref{sec:statefeed} after straightforward calculations. Observe that
\begin{equation}\label{varepsxdef}
 \varepsilon_x(z,t) = \mathcal{T}_o^{-1}[\tilde{\varepsilon}_x(t)](z) = \mathcal{T}_o^{-1}[\vartheta(t)](z) + \mathcal{T}_o^{-1}[\Gamma](z)\varepsilon_{\xi}(t)
\end{equation}
holds in view of \eqref{obstrafo} and \eqref{newcoordODE}. With this, one obtains the cascade 
\begin{subequations}\label{ccplant}
	\begin{align}
	\partial_te_x(z,t) &= \Lambda(z)\partial_ze_x(z,t) + H_0(z)e_{x_1}(0,t) \label{ccxeq}\\
	e_{x_2}(0,t) &= Q_0e_{x_1}(0,t)\label{ccuabc}\\
	e_{x_1}(1,t) &= Q_1E_2^T\mathcal{T}_o^{-1}[\vartheta(t)](1) + \textstyle\int_0^1K_x(z)\mathcal{T}_o^{-1}[\vartheta(t)](z)\d z + \widetilde{Q}_1\varepsilon_{\xi}(t)\label{ccaktbc}\\
	\dot{\xi}(t) &= (F-BK)\xi(t) + Be_{x_1}(0,t)\label{ccplantode}\\
	\partial_t\vartheta(z,t) &= \Lambda(z)\partial_z\vartheta(z,t) \label{copdes13}\\	                       
	\vartheta_{2}(0,t) &= 0\label{cdobsue23}\\
	\vartheta_{1}(1,t) &= Q_1\vartheta_{2}(1,t) - \textstyle\int_0^1S(\zeta)\vartheta(\zeta,t)\d\zeta\label{cobsrb2e23}\\
	\dot{\varepsilon}_{\xi}(t) &= (F - L_{\xi}E_1^{T}\Gamma(0))\varepsilon_{\xi}(t) - L_{\xi}\vartheta_1(0,t)\label{codetildobsfin}
	\end{align}
\end{subequations}
with
\begin{equation}
 \widetilde{Q}_1 = Q_1E_2^T\mathcal{T}_o^{-1}[\Gamma](1) + K_{\xi} + \textstyle\int_0^1K_x(z)\mathcal{T}_o^{-1}[\Gamma](z)\d z
\end{equation}
of two asymptotically stable subsystems \eqref{ccxeq}--\eqref{ccplantode} and \eqref{copdes13}--\eqref{codetildobsfin} (see \eqref{pde-ode cascade} and \eqref{obserrsys3}). Hence, the corresponding solution remains bounded for $t \in [0,t_o)$ so that \eqref{ccplant} simplifies to
\begin{subequations}\label{cccplant}
	\begin{align}
	\partial_te_x(z,t) &= \Lambda(z)\partial_ze_x(z,t) + H_0(z)e_{x_1}(0,t) \label{ccxeq2}\\
	e_{x_2}(0,t) &= Q_0e_{x_1}(0,t)\label{ccuabc2}\\
	e_{x_1}(1,t) &= \widetilde{Q}_1\varepsilon_{\xi}(t)\label{ccaktbc2}\\
	\dot{\xi}(t) &= (F-BK)\xi(t) + Be_{x_1}(0,t)\label{ccplantode2}\\
	\dot{\varepsilon}_{\xi}(t) &= (F - L_{\xi}E_1^{T}\Gamma(0))\varepsilon_{\xi}(t)\label{codetildobsfin2}
	\end{align}
\end{subequations}
for $t \geq t_o$, because $\vartheta(z,t) \equiv 0$, $t \geq t_o$ (see Section \ref{sec:obs}). In order to determine the solution of \eqref{ccxeq2}--\eqref{ccaktbc2} the ODE subsystem \eqref{codetildobsfin2} is decoupled from the PDE subsystem \eqref{ccxeq2}--\eqref{ccaktbc2}. For this, introduce the \emph{decoupling coordinates}
\begin{equation}\label{fdecoupl}
 \tilde{e}(z,t) = e_x(z,t) -\Sigma(z)e_{\xi}(t),
\end{equation}
in which $\Sigma(z) \in \mathbb{R}^{n \times n_{\xi}}$ has to be determined. By a time differentiation of \eqref{fdecoupl} and inserting \eqref{cccplant} in the result it is not difficult to show that $\Sigma(z)$ has to be the solution of the \emph{decoupling equations}
\begin{subequations}\label{odepdebvpc}
	\begin{align}
	\Lambda(z) \d_z\Sigma(z) &= \Sigma(z)(F - L_{\xi}E_1^{T}\Gamma(0)) - H_o(z)E_1^T\Sigma(0),\quad z \in (0,1)\label{pdedecouplst}\\
	(E_2^T-Q_0E_1^T)\Sigma(0) &= 0\\
	E_1^T\Sigma(1) &= \widetilde{Q}_1
	\end{align}
\end{subequations}
so that \eqref {fdecoupl} maps \eqref{cccplant} into
\begin{subequations}\label{ccccplant}
	\begin{align}
	\partial_t\tilde{e}(z,t) &= \Lambda(z)\partial_z\tilde{e}(z,t) + H_0(z)\tilde{e}_{1}(0,t) \label{cccxeq}\\
	\tilde{e}_{2}(0,t) &= Q_0\tilde{e}_{1}(0,t)\label{cccuabc}\\
	\tilde{e}_{1}(1,t) &= 0\label{ccaktbc3}\\
	\dot{\xi}(t) &= (F-BK)\xi(t) + BE_1^T\Sigma(0)e_{\xi}(t) + B\tilde{e}_{1}(0,t)\label{cccplantode}\\
	\dot{\varepsilon}_{\xi}(t) &= (F - L_{\xi}E_1^{T}\Gamma(0))\varepsilon_{\xi}(t).\label{ccodetildobsfin}
	\end{align}
\end{subequations}
The next theorem asserts the solvability of \eqref{odepdebvpc}.
\begin{theorem}[Decouplability of the PDE subsystem]\label{th:pdesub}
The BVP \eqref{odepdebvpc} has a unique solution $\Sigma(z) = [\Sigma_{ij}(z)] \in \mathbb{R}^{n \times n_{\xi}}$ with $\Sigma_{ij}$ piecewise $C^1$-functions, $i = 1,2,\ldots,n$, $j = 1,2,\ldots,n_{\xi}$.
\end{theorem}
\begin{proof}
Define $\widetilde{F} = F - L_{\xi}E_1^{T}\Gamma(0)$ and use the same approach as in the proof of Theorem \ref{lem:odepdecasc} to obtain for \eqref{odepdebvpc} the result
\begin{subequations}\label{ivpN1e2}
	\begin{align}
	\d_zE_1^T\Sigma(z)  &= \Lambda^{-1}_1(z)(E_1^T\Sigma(z)\widetilde{F} - H_1(z)E_1^T\Sigma(0)), \quad z \in (0,1)\label{N1odee2}\\
	E_1^T\Sigma(1) &= \widetilde{Q}_1\label{rbcdecplcls}
	\end{align}
\end{subequations}
and 
\begin{subequations}\label{ivpN2e2}
	\begin{align}
	\d_zE_2^T\Sigma(z) &= \Lambda^{-1}_2(z)(E_2^T\Sigma(z)\widetilde{F} - H_2(z)E_1^T\Sigma(0)), \quad z \in (0,1)\label{N2odee2}\\
	E_2^T\Sigma(0) &= Q_0E_1^T\Sigma(0)\label{N2bce2}
	\end{align}
\end{subequations}
in view of the partition \eqref{H0def}. For solving \eqref{ivpN1e2} define $\nu_i^T(z) = e_i^TE_1^T\Sigma(z)$, $i = 1,2,\ldots,p$, $h_i^T(z) = -\frac{1}{\lambda_{i}(z)}e_i^TH_1(z)$ and $\tilde{q}_i^T = e_i^T\widetilde{Q}_1$. This gives for \eqref{ivpN1e2} the result
\begin{subequations}\label{ivpN2setme2}
	\begin{align}
	\d_z\nu_i^T(z) &= \textstyle\frac{1}{\lambda_{i}(z)}\nu_i^T(z)\widetilde{F}\label{N2odeme2}
	+ h_i^T(z)E_1^T\Sigma(0), \quad z \in (0,1)\\
	\nu_i^T(1) &= \tilde{q}_i^T.\label{N2bcme2}
	\end{align}
\end{subequations}
With the \emph{fundamental matrix} 
\begin{equation}\label{fmatdefest}
\Psi_i(z,\zeta) = \text{e}^{\widetilde{F}(\phi_i(z) - \phi_i(\zeta))}, \quad i = 1,2,\ldots,n,
\end{equation}
and $\phi_i$ defined in \eqref{phidef}, the solution of \eqref{ivpN2setme2} is
\begin{equation}\label{sol1bvpcls}
\nu_i^T(z) = \nu_i^T(0)\Psi_{i}(z,0)
+ \textstyle\int_0^zh_i^T(\zeta)E_1^T\Sigma(0)\Psi_{i}(z,\zeta)\d\zeta.
\end{equation}
Inserting this in \eqref{rbcdecplcls} and defining
\begin{equation}
 M_{ij} = -\textstyle\int_0^1\Psi_i(1,\zeta)\frac{1}{\lambda_i(\zeta)}h_{ij}(\zeta)\d\zeta
\end{equation}
yields
\begin{multline}\label{glsclspre}
 \begin{bmatrix}
 \nu_1^T(0)\Psi_1(1,0)\\
 \vdots\\
 \nu_p^T(0)\Psi_p(1,0)
 \end{bmatrix}\\
 -
 \int_0^1\begin{bmatrix}
 0^T\\
 \nu_1^T(0)\Psi_2(1,\zeta)\frac{1}{\lambda_2(\zeta)}h_{21}(\zeta)\\
 \vdots\\
 \nu_1^T(0)\Psi_p(1,\zeta)\frac{1}{\lambda_p(\zeta)}h_{p1}(\zeta) + \ldots + \nu_{p-1}^T(0)\Psi_p(1,\zeta)\frac{1}{\lambda_p(\zeta)}h_{p\,p-1}(\zeta)\\
 \end{bmatrix}\d\zeta = \widetilde{Q}_1
\end{multline}
after a simple calculation utilizing \eqref{h01def}. This can be rearranged in the form
\begin{equation}\label{nu0solv2}
\begin{bmatrix}
\nu_1^T(0) & \ldots & \nu_p^T(0)
\end{bmatrix}M = \begin{bmatrix}
\tilde{q}_1^T & \ldots & \tilde{q}_p^T
\end{bmatrix},
\end{equation}
in which
\begin{equation}\label{Mdef2}
M = \begin{bmatrix}
\Psi_1(1,0) & M_{21} & \ldots       & M_{p-1\,1}       & M_{p1}\\
 0          & \ddots & \ddots       & \ddots          & \vdots\\
 \vdots     & \ddots & \ddots       & \ddots          & \vdots\\
 \vdots     & \ddots & \ddots       & \Psi_{p-1}(1,0) & M_{p\,p-1}\\
 0	        & \ldots & \ldots       & 0               & \Psi_p(1,0)
\end{bmatrix}
\end{equation}
holds. From this,
\begin{equation}
\det M = \prod_{i=1}^p\det\Psi_i(1,0) = \prod_{i=1}^p\det \text{e}^{\widetilde{F}\phi_i(1)} 
= \prod_{i=1}^p \text{e}^{\operatorname{tr}(\widetilde{F}\phi_i(1))} \neq 0
\end{equation}	
follows. Hence, the initial values $\nu_i^T(0)$, $i = 1,2,\ldots,p$, can be uniquely determined so that the solution of \eqref{ivpN2setme2} is given by \eqref{sol1bvpcls}. Similarly, by defining $\mu_i^T(z) = e_i^TE_2^T\Sigma(z)$, $i = 1,2,\ldots,m$, the solution of \eqref{ivpN2e2} is given by
\begin{equation}
\mu_i^T(z) = e_i^TQ_0E_1^T\Sigma(0)\Psi_{i+p}(z,0)
- \textstyle\int_0^z\frac{1}{\lambda_{i+p}(\zeta)}e_i^TH_2(\zeta)E_1^T\Sigma(0)\Psi_{i+p}(z,\zeta)\d\zeta.
\end{equation}
Since the elements of $H_0(z)$ in \eqref{pdedecouplst} are piecewise $C^1$-functions (see Section \ref{sec:soldecoupl}) the same is true for the functions appearing in the right hand sides of \eqref{ivpN1e2} and \eqref{ivpN2e2}. As a consequence, \eqref{odepdebvpc} has a unique piecewise $C^1$-solution.
\end{proof}	
As \eqref{cccxeq}--\eqref{ccaktbc3} coincides with \eqref{casdpde}--\eqref{cascrb2} one obtains for $t \geq t_o + t_c$
\begin{equation}\label{fsys}
 	\begin{bmatrix}\dot{\xi}(t)\\
                	\dot{\varepsilon}_{\xi}(t)
    \end{bmatrix}
 	= \begin{bmatrix}
 	 F-BK & BE_1^T\Sigma(0)\\
     0    & F - L_{\xi}E_1^{T}\Gamma(0)
 	\end{bmatrix}
 	\begin{bmatrix}
 	 \xi(t)\\
 	 \varepsilon_{\xi}(t)
 	\end{bmatrix},
\end{equation}
because $\tilde{e}(z,t) \equiv 0$, $t \geq t_o + t_c$ (see Section \ref{sec:pdeodedecoupl}) implying
\begin{equation}\label{exsimple}
 e_x(z,t) = \Sigma(z)\varepsilon_{\xi}(t), \quad t \geq t_o + t_c
\end{equation}
in view of \eqref{fdecoupl}. Obviously, the system \eqref{fsys} is asymptotically stable, because $F-BK$ and $F - L_{\xi}E_1^{T}\Gamma(0)$ are Hurwitz matrices by assumption (see Sections \ref{sec:pdeodedecoupl} and \ref{sec:obs}). From this, the result
\begin{equation}\label{epsxas}
 \lim_{t \to \infty}\varepsilon_x(z,t) = \lim_{t \to \infty}\mathcal{T}_o^{-1}[\Gamma](z)\varepsilon_{\xi}(t) = 0
\end{equation}
can be inferred, when taking \eqref{varepsxdef} and $\vartheta(z,t) \equiv 0$, $t \geq t_o$ into account. Furthermore, 
\begin{equation}\label{xtilxas}
 \lim_{t \to \infty}\tilde{x}(z,t) =   \lim_{t \to \infty}(\mathcal{T}^{-1}_2[\Sigma](z)\varepsilon_{\xi}(t) + N_I(z)\xi(t)) = 0
\end{equation}
follows from \eqref{cccord2} and \eqref{exsimple}. This directly leads to
\begin{equation}\label{xas}
 \lim_{t \to \infty}x(z,t) = \lim_{t \to \infty}\mathcal{T}_1^{-1}[\tilde{x}(t)](z) = 0
\end{equation}
in light of \eqref{btrafocl}. Hence, the closed-loop system \eqref{plant} and \eqref{obsc} represented by \eqref{plant} and the observer error dynamics \eqref{obse} is asymptotically stable. This, of course, also implies the asymptotic stability of \eqref{plant} and \eqref{obsc}, which is the result of the next theorem.
\begin{theorem}[Closed-loop stability]\label{th:pdesub}
Assume that $F-BK$ and $F - L_{\xi}E_1^{T}\Gamma(0)$ are Hurwitz matrices. Then, the closed-loop system resulting from applying \eqref{obsc} to \eqref{plant} is asymptotically stable for all piecewise continuous IC $x(z,0) = x_0(z)$ and $\hat{x}(z,0) = \hat{x}_0(z)$ as well as for all $\xi(0), \hat{\xi}(0) \in \mathbb{R}^{n_{\xi}}$. 
\end{theorem}
\begin{rem}
It is worth noticing that the closed-loop dynamics for $t \geq t_o + t_c$ can be specified by the eigenvalue assignment for the matrices $F-BK$ and $F - L_{\xi}E_1^{T}\Gamma(0)$. This directly follows from \eqref{fsys} and \eqref{epsxas}--\eqref{xas}. Thereby, \eqref{fsys} indicates that the corresponding eigenvalues can be independently specified for the closed-loop system on the basis of the state feedback and observer design in the previous sections. Thus, in this sense, the \emph{separation principle} is valid for the design of the observer-based compensator \eqref{obsc}. \hfill $\triangleleft$
\end{rem}

\section{Example}
While a multitude of technological processes mentioned in the introduction could serve to illustrate the results presented here, for simplicity, only a numerical example is used to confirm the effectiveness of the observer-based compensator in \eqref{obsc}, i.e., the backstepping controller \eqref{sfeed} combined with the observer \eqref{obs}.

The unstable $4\times4$ heterodirectional hyperbolic system under consideration has the distributed state $x(z,t)\in\mathbb R^4$ with $x_1=[x^1,x^2]^T$ describing the values propagating in the negative $z$-direction and $x_2=[x^3,x^4]^T$ in the 
opposite direction. The system is given in the form \eqref{xeq} with matrices
\begin{align}
\Lambda = \begin{bmatrix} 3 & 0 & 0 & 0 \\[.5ex] 0 & 2 & 0 & 0 \\[.5ex] 0 & 0 & -1 & 0 \\[.5ex] 0 & 0 & 0 & -2 \end{bmatrix}, \qquad
A(z) = \begin{bmatrix} 0 & 2 \mathrm{e}^{2z} & \mathrm{e}^{3z} \sin z & -2 \mathrm{e}^{2z} \\[.5ex] -3  \mathrm{e}^{-2z} & 0 &  \mathrm{e}^{z} & 2 \\[.5ex] -2 \mathrm{e}^{-3z} &  \mathrm{e}^{-z} & 0 & z \mathrm{e}^{-z} \\[.5ex]  \mathrm{e}^{-2z} & -1 & -2 \mathrm{e}^{z} & 0 \end{bmatrix}
\end{align}
and $C_1=0$. The matrix $Q_1$ in \eqref{aktbc} is defined as
\begin{equation}
Q_1 = \begin{bmatrix} 2\mathrm{e}^{3} & \mathrm{e}^{2} \\[.5ex] \mathrm{e}^{1} & 2 \end{bmatrix}
\end{equation}
and the dynamic BCs
\begin{subequations}
	\label{eq:ex_dynbc}
	\begin{align}
	\ddot{x}^3(0,t) - 2x^3(0,t) &= x^1(0,t) - 3x^2(0,t) \\
	\dot{x}^4(0,t) &= -\dot{x}^3(0,t) + 2\dot{x}^1(0,t) + x^1(0,t)
	\end{align}
\end{subequations}
appear at $z=0$, which results in a coupled PDE-ODE system. As \eqref{eq:ex_dynbc} is a third-order ODE w.r.t.\ the input $x_1(0,t)$, introduce the lumped state
\begin{equation}
\label{eq:ex_state}
\xi(t) = \begin{bmatrix} x^3(0,t) \\[.5ex] \dot x^3(0,t) \\[.5ex] x^4(0,t)-2x^1(0,t) \end{bmatrix}.
\end{equation}
Consequently, the dynamic BCs \eqref{eq:ex_dynbc} can be written in the form \eqref{uabc} and \eqref{plantode}, where the matrices $F$ and $B$ in \eqref{plantode} follow from substituting \eqref{eq:ex_state} into \eqref{eq:ex_dynbc}, and the matrices $C_2$ and $Q_0$ from solving \eqref{eq:ex_state} for $x_2(0,t)$, giving
\begin{equation}
F = \begin{bmatrix} 0 & 1 & 0 \\[.5ex] 2 & 0 & 0 \\[.5ex] 0 & -1 & 0 \end{bmatrix}, \qquad
B = \begin{bmatrix} 0 & 0 \\[.5ex] 1 & -3 \\[.5ex] 1 & 0 \end{bmatrix}, \qquad
C_2 = \begin{bmatrix} 1 & 0 & 0 \\[.5ex] 0 & 0 & 1 \end{bmatrix}, \qquad
Q_0 = \begin{bmatrix} 0 & 0 \\[.5ex] 2 & 0 \end{bmatrix}.
\end{equation}

As the numerical example given is unstable, a state feedback controller \eqref{sfeed} is necessary to stabilize the system. For that, the design parameters in \eqref{cdbvp3aa}--\eqref{cdbvp2aa}, introduced for well-posedness of the kernel equations \eqref{ckbvp} (cf.\ Appendix \ref{appA}), are chosen such that \eqref{ckbvp} has a piecewise $C^1$-solution $K(z,\zeta)$. As the pair $(F,B)$ is controllable, the gain matrix $K$ is determined such that the eigenvalues of $F-BK$ in \eqref{cascode} are $-2$, $-3$ and $-4$. Based on that, the solution of the $p^2=4$ Volterra integral equations from Section \ref{sec:p2} completes the feedback design.
\begin{figure}[t]
	\centering
	\setlength\figureheight{6cm}
	\setlength\figurewidth{0.9\linewidth}
	\includegraphics[width=0.9\linewidth]{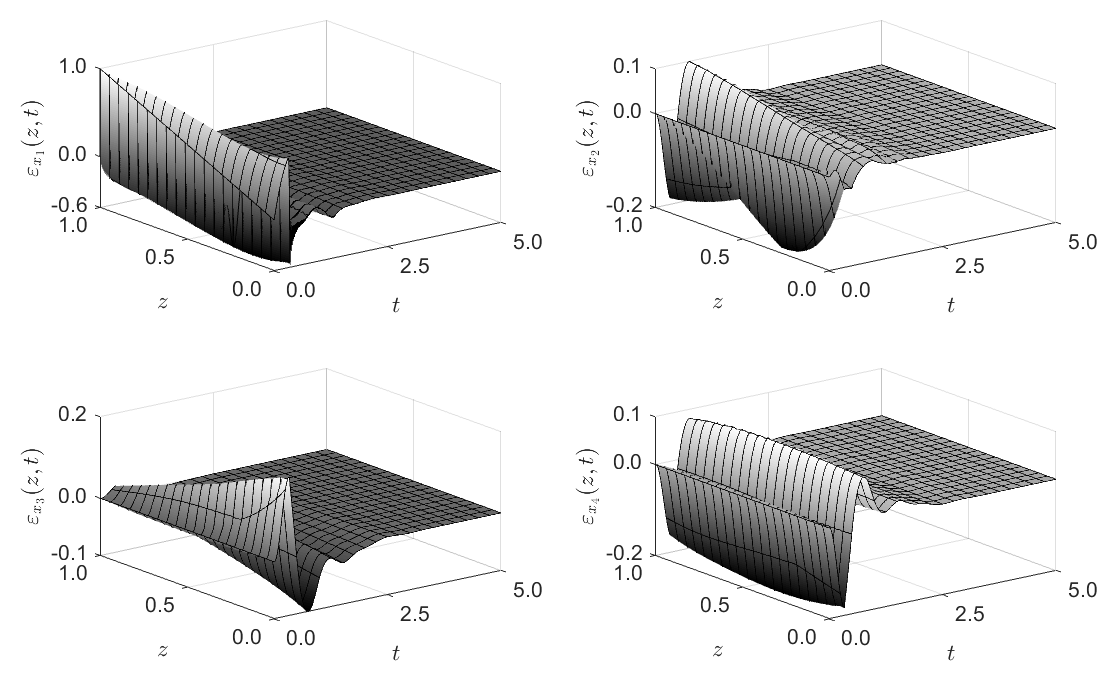}
	\caption{State profiles of the distributed observer errors $\varepsilon_{x_i}(z,t)= x_i(z,t) - \hat{x}_i(z,t)$, $i=1,\dots,4$.}
	\label{fig:surface_x_err}
\end{figure}
\begin{figure}[t]
	\centering
	\setlength\figureheight{4cm}
	\setlength\figurewidth{0.9\linewidth}
	\includegraphics[width=0.9\linewidth]{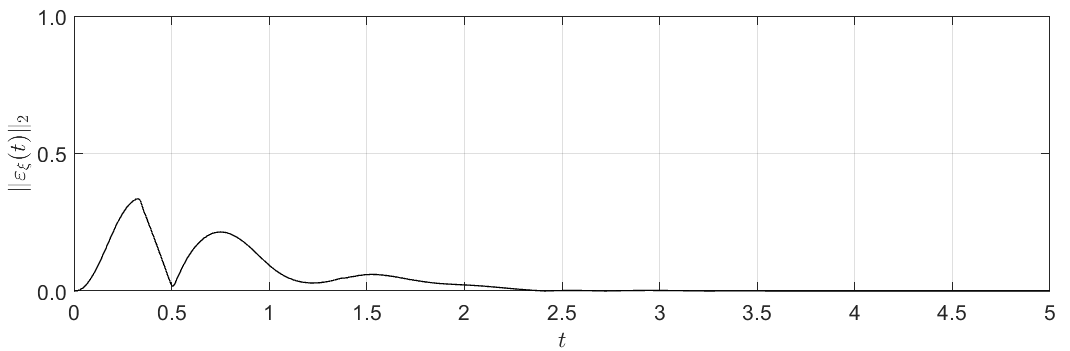}
	\caption{Euclidean norm $\norm{\varepsilon_\xi(t)}_2=\norm{\xi(t)-\hat{\xi}(t)}_2$ of the lumped observer errors.}
	\label{fig:norm_xi_err}
\end{figure}

\begin{figure}[t]
	\centering
	\setlength\figureheight{6cm}
	\setlength\figurewidth{0.9\linewidth}
	\includegraphics[width=0.9\linewidth]{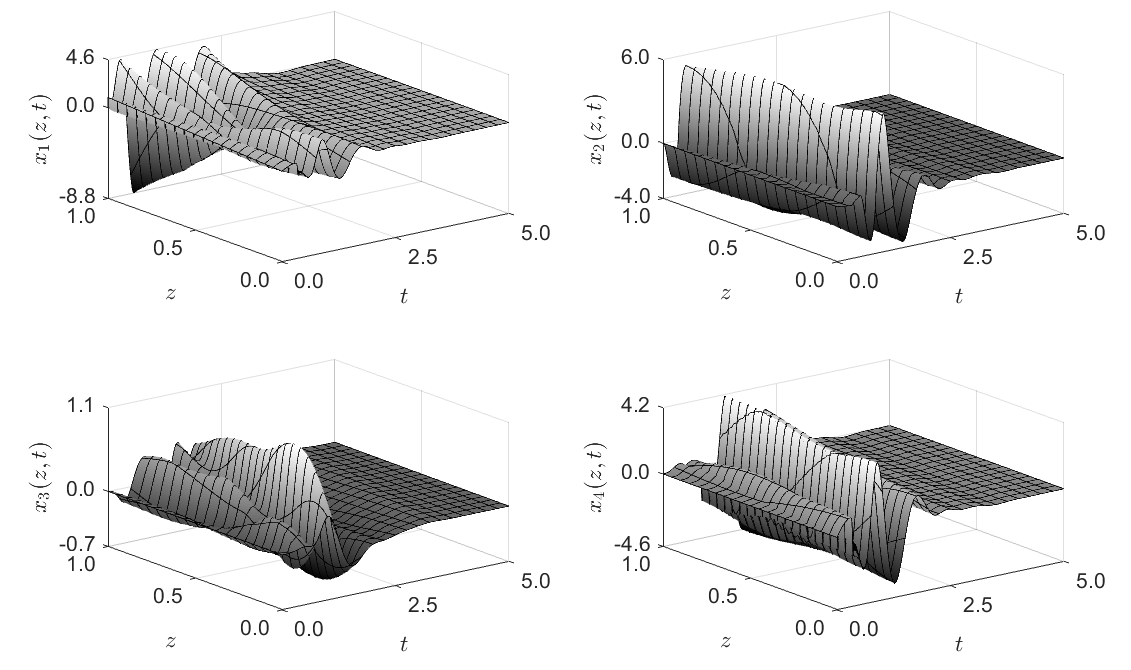}
	\caption{State profiles of the plant's distributed state components $x_i$, $i=1,\dots,4$ in the closed-loop system.}
	\label{fig:surface_x}
\end{figure}
\begin{figure}[t]
	\centering
	\setlength\figureheight{4cm}
	\setlength\figurewidth{0.9\linewidth}
	\includegraphics[width=0.9\linewidth]{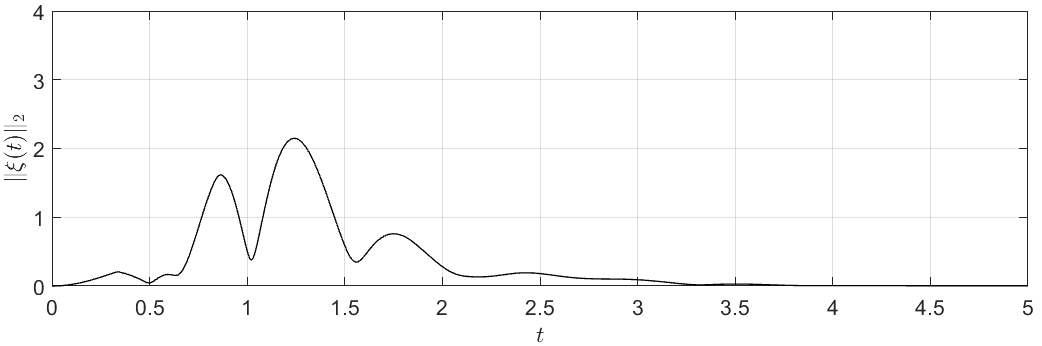}
	\caption{Euclidean norm $\norm{\xi(t)}_2$ of the plant's lumped state in the closed-loop system.}
	\label{fig:norm_xi}
\end{figure}

The implementation of the feedback controller \eqref{sfeed} requires an observer as only the output $y(t)=x_1(0,t)$ is measured. According to Theorem \ref{lem:obs}, the design of an observer \eqref{obs} requires  observability of the pair $(E_1^{T}\Gamma(0),F)$. Setting the design parameters involved in the solution of the kernel equations \eqref{obsbvp} such that $R_I(z,\zeta)$ is piecewise $C^1$, based on the numerical solution of \eqref{odepdebvp} (see Theorem \ref{lem:odepdecasc}) and
\begin{equation}
E_1^T\Gamma(0) = \begin{bmatrix} 49.6416 & -16.6636 & 31.9830 \\ -2.9637 & 3.5913 & 3.2244 \end{bmatrix}
\end{equation}
the observability condition is verified to be met. Consequently, the gain matrix $L_\xi$ in the state observer is chosen such that the eigenvalues of $F - L_{\xi}E_1^{T}\Gamma(0)$ in \eqref{odetildobsfin} are $-5$, $-6$ and $-7$. Then, the observer gain $L(z)$ follows from \eqref{obsgainf}.

The simulation results of the observer-based state feedback provided in Figures \ref{fig:surface_x_err}--\ref{fig:norm_xi} were obtained using an initialization of the unstable plant with $x(z,0) = [z,0,0,0]^T$ and $\xi(0) = 0$, while the ICs of the observer were set to $\hat{x}(z,0) = 0$ and $\hat{\xi}(0) = 0$. The profiles of the distributed observer errors in Fig.~\ref{fig:surface_x_err} and the Euclidean norm $\norm{\varepsilon_\xi(t)}_2$ of the lumped observer error in Fig.~\ref{fig:norm_xi_err} confirm that the observer provides correct estimates.  The lumped state error dynamics are autonomous for $t_o=\frac{1}{3}+\frac{1}{2}+1$ as defined in \eqref{eq:to} and, by that, $\varepsilon_\xi$ exponentially converges to zero for $t\ge\frac{11}{6}$. Following this minimal settling time, the observer errors are almost zero for $t>2.5$, which is confirmed by the simulation results in Figs.~\ref{fig:surface_x_err} and \ref{fig:norm_xi_err}. The minimal settling time for the plant state of the closed-loop system is $t_c+t_o=\frac{22}{6}$ (cf.\ \eqref{eq:tc} and \eqref{eq:to}), after which time the closed-loop dynamics is due only to the ODE \eqref{fsys}. Figs.~\ref{fig:surface_x} and \ref{fig:norm_xi} depict the distributed and the lumped state of the plant in the closed-loop system. It can be seen that the output feedback controller manages to asymptotically stabilize the unstable open-loop system, with the states almost being zero for $t>3.8$.

\section{Concluding remarks}
An interesting topic for further research is the extension of the observer design to  collocated measurements. This is a challenging problem, because the ODE in the observer is then subject to a coupling with the PDE from both boundaries. Similarly, this type of coupling can also arise in the PDE-ODE system itself, if, for example, dynamic boundary conditions are present at both boundaries. Then, the resulting PDE-ODE system is driven by an ODE, which hinders the controller design.

\bibliographystyle{apacite}
\bibliography{mybib}
%

\appendix
\section{Solution of the state feedback kernel equations}\label{appA}
In the following, the approach in \cite{Hu15b} for solving the kernel equations \eqref{ckbvp} is shortly reviewed in order to make the article selfcontained. Premultiplying \eqref{cdbvp3} by $E_1E_1^T + E_2E_2^T = I_{n}$ yields $(E_1E_1^T + E_2E_2^T)K(z,0)\Lambda(0)(E_1 + E_2Q_0)
= (E_1E_1^T + E_2E_2^T)A_0(z)$, from which
\begin{equation}\label{a:rbpre}
E_1^TK(z,0)\Lambda(0)(E_1 + E_2Q_0) = E_1^TA_0(z)
\end{equation}
and
\begin{equation}\label{a:rbpre2}
E_2^TK(z,0)\Lambda(0)(E_1 + E_2Q_0) = E_2^TA_0(z)
\end{equation}
can be deduced.  The lower triangular structure of $A_{1}(z) = E_1^TA_{0}(z)$ in \eqref{a01def} implies $[E_1^TA_0(z)]_{i \leq j} = 0$. With this, the result
\begin{equation}
[E_1^TK(z,0)\Lambda(0) (E_1 + E_2Q_0)]_{i \leq j} = 0
\end{equation}
follows from applying $[\,\cdot\,]_{i \leq j}$ to \eqref{a:rbpre}, which are BCs to be fulfilled by the kernel $K(z,\zeta)$. Accordingly, the elements $a_{ij}(z)$ in $A_{1}(z)$ are obtained from \eqref{a:rbpre} with
\begin{equation}\label{appcoeffcomp}
[E_1^TA_0(z)]_{i > j} \!=\! [E_1^TK(z,0)\Lambda(0) (E_1 + E_2Q_0)]_{i > j}.
\end{equation} 
The matrix 
\begin{equation}\label{appcoeffcomp2}
A_{2}(z) \!=\! E_2^TA_0(z) \!=\! E_2^TK(z,0)\Lambda(0) (E_1 + E_2Q_0)
\end{equation}
in \eqref{A0def} is directly implied by \eqref{a:rbpre2}. Consequently, the kernel equations \eqref{ckbvp} take the form 
\begin{subequations}\label{ckbvpa}
	\begin{align}
	\Lambda(z)\partial_zK(z,\zeta) + \partial_{\zeta}(K(z,\zeta)\Lambda(\zeta)) &= K(z,\zeta)A(\zeta)\label{cdbvp1a}\\
	[E_1^TK(z,0)\Lambda(0)(E_1 + E_2Q_0)]_{i \leq j} &=
	0\label{cdbvp3a}\\
	K(z,z)\Lambda(z) - \Lambda(z) K(z,z) &= A(z).\label{cdbvp2a}
	\end{align}
\end{subequations}
For $n \geq 3$, the considered class of hyperbolic systems \eqref{plant} contains at least two transport equations propagating in the same direction so that so-called \emph{homodirectional kernel-PDEs} appear in \eqref{cdbvp1a}. As a consequence, there exist kernel PDEs, which need more than one BC in order to uniquely determine their solution on the triangular spatial domain $0 \leq \zeta \leq z \leq 1$  (for an illustration see \cite{Hu15a}). Therefore, the \emph{artificial BCs} 
\begin{subequations}\label{ckbvp2}
	\begin{align}
	[E_1^TK(1,\zeta)E_1]_{i > j} &= l_{ij}(\zeta)\label{cdbvp3aa}\\
	[E_2^TK(z,0)E_2]_{i \geq j} &= m_{ij}(z)\label{cdbvp1aa}\\
	[E_2^TK(1,\zeta)E_2]_{i < j} &= n_{ij}(\zeta)\label{cdbvp2aa}
	\end{align}
\end{subequations}
are introduced so that the resulting BVP admits a unique piecewise $C^1$-solution.  
Thereby, the functions $l_{ij}, m_{ij}, n_{ij} \in C^{\infty}[0,1]$ can be chosen arbitrarily.
Hence, they constitute additional design parameters that may be utilized to shape the feedback gains $K_{\xi}$, $K_x(z)$ and thus the transients of the closed-loop dynamics. Further details for solving the kernel BVP \eqref{ckbvpa} and \eqref{ckbvp2} with the method of successive approximations can be found in \cite{Cor13,An16}.

\section{Solution of the observer kernel equations}\label{appB}
Introduce the transformation of the independent variables
\begin{subequations}
	\begin{align}
	\xi &= 1 - \zeta\\
	\eta  &= 1 - z
	\end{align}  
\end{subequations}
and define $\tilde{R}_I(\xi,\eta) = R_I(1-\eta,1-\xi)$, $\tilde{\Lambda}(\xi) = \Lambda(1 - \xi) = \Lambda(\zeta)$, $\tilde{\Lambda}(\eta) = \Lambda(1 - \eta) = \Lambda(z)$, $\tilde{A}(\eta) = A^T(1 - \eta)$ and $\tilde{F}(\xi) = -F^T(1-\xi)$. With this, an easy calculation shows that the kernel equations \eqref{obsbvp} result in 
\begin{subequations}\label{ckbvpd}
	\begin{align}
	\partial_{\xi}(\tilde{\Lambda}(\xi)\tilde{R}_I^T(\xi,\eta)) + \partial_{\eta}\tilde{R}^T_I(\xi,\eta)\tilde{\Lambda}(\eta) &= \tilde{R}^T_I(\xi,\eta)\tilde{A}(\eta)\label{cdbvp1d}\\
	\tilde{R}^T_I(\xi,0)(E_1 - E_2Q_1^T) &= \tilde{F}(\xi)\label{cdbvp3d}\\
	\tilde{R}^T_I(\xi,\xi)\tilde{\Lambda}(\xi) - \tilde{\Lambda}(\xi) \tilde{R}^T_I(\xi,\xi) &= \tilde{A}(\xi)\label{cdbvp2d}
	\end{align}
\end{subequations}
with \eqref{cdbvp1d} defined on $0 < \eta < \xi < 1$. It can easily be verified that the BVP \eqref{ckbvpd} has the same structure as the BVP appearing in the proof of Theorem A.1  in \cite{Hu15b}. Then, the regularity of the kernel $\tilde{R}_I^T(\xi,\eta)$ is deducible from this result. Consequently, the equations to be solved for determining $\tilde{R}_I^T(\xi,\eta)$ are derivable with the same reasoning as in Appendix \ref{appA}. From this, the kernel $R_I(z,\zeta) = \tilde{R}_I(1-\zeta,1-z)$ results.

\end{document}